\documentclass[11pt,a4paper]{article}
\usepackage[T1]{fontenc}
\usepackage{lmodern}
\usepackage{amsmath,amssymb,amsthm,mathtools}
\usepackage[margin=27mm]{geometry}
\usepackage{microtype,booktabs,longtable,array,enumitem,seqsplit}
\usepackage{needspace}
\usepackage{xcolor}
\usepackage[colorlinks=true,linkcolor=blue!45!black,citecolor=blue!45!black,urlcolor=blue!45!black]{hyperref}
\hypersetup{pdftitle={Ehrhart Series of the Birkhoff Polytopes via Constant Terms and Finite-Field Evaluation}}
\newtheorem{theorem}{Theorem}[section]
\newtheorem{proposition}[theorem]{Proposition}
\newtheorem{lemma}[theorem]{Lemma}
\newtheorem{corollary}[theorem]{Corollary}
\theoremstyle{definition}

\theoremstyle{remark}
\newtheorem{remark}[theorem]{Remark}
\newcommand{\F}{\mathbb F}
\newcommand{\Z}{\mathbb Z}
\newcommand{\R}{\mathbb R}
\newcommand{\N}{\mathbb Z_{\geq0}}
\newcommand{\B}{\mathcal B}
\newcommand{\Ehr}{\operatorname{Ehr}}
\newcommand{\Vol}{\operatorname{Vol}_{\Z}}

\DeclareMathOperator*{\CT}{CT}
\newcommand{\supp}{\operatorname{supp}}
\newcommand{\mult}{\operatorname{mult}}
\newcommand{\calM}{\mathcal M}
\newcommand{\calD}{\mathcal D}

\setlist[enumerate]{itemsep=3pt,topsep=5pt}
\setlist[itemize]{itemsep=3pt,topsep=5pt}
\allowdisplaybreaks[2]
\title
{Exact Ehrhart Series of Birkhoff Polytopes \\ via Constant Terms and Finite-Field Evaluation}
\author{Xinru Jiang $^a$, Guoce Xin $^{b,*}$, Chen Zhang $^c$, and Yueming Zhong $^a$
\\[2mm]
{\small $^a$ School of Mathematics and Statistics, Hainan University,}\\[-0.8ex]
{\small Haikou 570228, PR China}\\
{\small $^b$ School of Mathematical Sciences, Capital Normal University, }\\[-0.8ex]
{\small Beijing, 100048, PR China}\\
{\small $^c$ Center for Combinatorics, LPMC, Nankai University,}\\[-0.8ex]
{\small Tianjin 300071, PR China}\\
{\small Email addresses: chiang\_11110@163.com (X. Jiang)}\\[-0.8ex]
{\small \hspace{5.5em} guoce\_xin@163.com (G. Xin)}\\[-0.8ex]
{\small \hspace{5em} ch\_enz@163.com (C. Zhang)}\\[-0.8ex]
{\small \hspace{10.5em} zhongyueming107@gmail.com (Y. Zhong)}\\[1.2ex]
$^*$ Corresponding author.
}

\date{\today}

\begin{document}
\maketitle

\begin{abstract}

The Ehrhart series of the $n$th Birkhoff polytope is $\sum_{r\geq0}H_n(r)z^r$, where $H_n(r)$ counts nonnegative integer $n\times n$ matrices whose row and column sums all equal $r$. We present an exact method for computing this series using constant terms and finite fields. A root filter expresses $H_n(r)$ as a weighted sum of the values $h_r(M)^n$, where $h_r$ is the complete homogeneous symmetric polynomial and $M$ ranges over multisets of $m$th roots of unity with $m=r+1$. Constant-term cancellation reduces the evaluation of $h_r(M)$ to a sum over repeated elements $a$ of $M$. For a particular $a$ of multiplicity $\mu_a$, the computation uses a generalized Todd coefficient of degree $\mu_a-2$. Sums of $h_r(M)^n$ over selected multiplicity classes are handled using symmetric function techniques. Together with the remaining individual evaluations, this gives $O_n(m^{n-5}+m^4)$ field operations for each admissible prime and fixed $n\geq5$. An explicit bound and the Chinese remainder theorem recover the integer counts, and Ehrhart symmetry determines the full series. The same method applies to the World Cup problem, which counts the same matrices with diagonal entries required to be $0$. We prove correctness and compute complete Birkhoff series through order $14$ and World Cup series through order $13$. The Birkhoff series for orders $10$--$14$ and the World Cup series for orders $9$--$13$ are tabulated in the appendices.

\end{abstract}
\noindent\textbf{Keywords.} Birkhoff polytope; Ehrhart series; complete
homogeneous symmetric polynomial; constant term; generalized Todd
polynomial; finite field.

\medskip
\noindent\textbf{2020 Mathematics Subject Classification.}
Primary 05A15; Secondary 05E05, 52B20.

\section{Introduction}

Let
\[
 \B_n=\bigl\{A=(a_{ij})\in\R_{\geq0}^{n\times n}:
 \sum_{j=1}^n a_{ij}=1,\quad\sum_{i=1}^n a_{ij}=1\bigr\}
\]
be the $n$th Birkhoff polytope. Its Ehrhart counting function is
\[
 H_n(r)=\#\bigl\{A=(a_{ij})\in\N^{n\times n}:
 \sum_j a_{ij}=r,\quad\sum_i a_{ij}=r\bigr\},\qquad r\in\N.
\]
In particular, $H_n(0)=1$ and $H_n(1)=n!$. This is the classical problem of counting semimagic squares.

The well-known Birkhoff--von Neumann theorem states that $\B_n$ is the convex hull of the permutation matrices~\cite{Birkhoff46,vonNeumann53}. The formulas $H_1(r)=1$ and $H_2(r)=r+1$ are immediate. In his 1915 treatise, MacMahon~\cite{MacMahon} obtained the first nontrivial formula for $H_3(r)$. In 1966, Anand, Dumir, and Gupta~\cite{ADG} conjectured the polynomiality and reciprocity properties of these counts. Stanley~\cite{Stanley73} and Ehrhart~\cite{Ehrhart73} independently proved these conjectures in 1973.

Several methods have been developed for the exact computation of
$H_n(r)$. In 1999, Chan and Robbins~\cite{CR} computed the complete
polynomial $H_8(r)$ and the volume of $\B_8$. In 2003, Beck and
Pixton~\cite{BP} developed a residue method and computed $H_9(r)$;
their accompanying data include volumes through order
$10$~\cite{BPdata}. In 2009, De Loera, Liu, and
Yoshida~\cite{DLY} gave a multivariate generating function for
semimagic squares and derived formulas for Ehrhart coefficients
and the volumes of Birkhoff polytopes and their faces.
In 2025, Xin, Zhang, Zhou, and Zhong~\cite{XZZZ25} introduced
the constant term algebra of type $A$, motivated by an iterated
Laurent series treatment of Beck and Pixton's residue formulas,
and described its structure using a basis indexed by forests.

Our main contributions are as follows.
\begin{enumerate}
\item We compute the complete Ehrhart series of $ \B_n $ for $ n\le 14 $, together with their normalized volumes. The series for $ \B_{10},\B_{11},\B_{12},\B_{13},\B_{14} $ are new and are listed in Appendix~\ref{app:birkhoff-series}.

\item We develop a constant-term and Todd-polynomial method for these computations, prove its correctness, and analyze its arithmetic complexity.

\item We apply the method to the World Cup problem
\[
 \mathcal W_n=\{A\in\B_n:a_{ii}=0\text{ for all }i\}
\]
introduced by Ekhad and Zeilberger~\cite{EZ}, and obtain the complete series through order $13$ (Appendix~\ref{app:worldcup-series}).
\end{enumerate}

Write $ h_r(\mathbf X) $ for the complete homogeneous symmetric polynomial in $ \mathbf X=(x_1,\ldots,x_n) $. Then $ H_n(r) $ is the coefficient of $ x_1^r\cdots x_n^r $ in $ h_r(\mathbf X)^n $. By homogeneity, this coefficient can be written as a finite average over $ m $th roots of unity for any $ m>r $, a special case of Karasev and Petrov's coefficient formula~\cite{KP}. We use $ m=r+1 $. Permutation and scaling symmetries reduce the sum to a weighted sum over multisets. The remaining task is to evaluate $ h_r $ on these multisets.

 We express each evaluation as a univariate constant term. The partial-fraction decomposition then has contributions only from repeated roots. This use of selected partial-fraction terms also appears in Xin's partition-analysis algorithms~\cite{Xin04,Xin15}. A change of variables due to Liu and Xin~\cite[Proposition~2.4]{LX26} expresses each contribution as a generalized Todd coefficient, computed by the methods of Xin, Zhang, and Zhang~\cite{XZZ25}. A repeated root of multiplicity $ \mu\ge2 $ requires only a Todd coefficient of degree $ \mu-2 $, independent of $ r $. This leads to an algorithm whose complexity is
 $O_n(m^{n-5}+m^4)$ field operations.

Our main counting formula, Theorem~\ref{thm:anchor}, expresses
$H_n(r)$ modulo a suitable prime as a weighted sum over root
multisets in which $1$ has multiplicity at least $2$.
Proposition~\ref{thm:todd-local} evaluates each summand through
generalized Todd coefficients, and Corollary~\ref{cor:max-anchor}
reduces the sum further by marking a root of maximal multiplicity.
The marking argument gives the correct weights even when some
scaling orbits have nontrivial stabilizers. These results determine
the count in a finite field. A combinatorial upper bound then
allows its integer value to be recovered by the Chinese remainder
theorem. Ehrhart symmetry determines the complete series from
finitely many counts.

For the World Cup problem, Ekhad and Zeilberger~\cite{EZ} computed the counting polynomials through order $6$. Zeilberger offered a prize for orders $7$ through $11$. In 2026, the cases $7$ and $8$ were solved by Xin and Zhang and, independently, by Vandevere; see the website attached to~\cite{EZ}. Our computations give the complete series through order $13$ and agree with the known results through order $8$. The series for $9\le n\le13$ are listed in Appendix~\ref{app:worldcup-series}. For this problem, the polynomial for row $i$ is $ h_r(\mathbf{X}\setminus\{x_i\})=h_r(\mathbf{X})-x_i h_{r-1}(\mathbf{X})$, where one occurrence of $x_i$ is removed. The two values use the same Todd coefficients, with one additional convolution for $h_{r-1}$. Finite products sum five nonzero multiplicity classes, and one further class vanishes. Together with individual evaluation of the remaining multisets, this gives $O_n(m^{n-4}+m^4)$ field operations for fixed $n\ge4$.

\Needspace{7\baselineskip}
The remainder of the paper is organized as follows.
Section~\ref{sec:method} gives the constant-term method for
computing the Birkhoff Ehrhart series, including the main counting
formula, symmetry reduction, and class summation.
Section~\ref{sec:algorithms} presents the algorithms and their
complexity, an order-$3$ example, and computational results,
including normalized volumes. Section~\ref{sec:worldcup} applies
the method to the World Cup problem.
Section~\ref{sec:conclusion} contains the concluding remarks.
Appendices~\ref{app:birkhoff-series} and~\ref{app:worldcup-series}
give the complete Birkhoff series for orders $10$--$14$ and
the World Cup series for orders $9$--$13$, respectively.

\section{A constant-term method for Birkhoff Ehrhart series}\label{sec:method}

The computation of the Ehrhart series of $\B_n$ proceeds as follows.
Ehrhart symmetry reduces the required data to finitely many values of
$H_n(r)$.
The row generating polynomial expresses each count as a coefficient
of a symmetric polynomial, and a finite-field root filter turns
that coefficient into an average. Permutation and scaling
symmetries reduce the outer sum. We then express each evaluation
as local constant terms, compute them by generalized Todd
coefficients, and sum low-multiplicity classes through finite
subset products.

\subsection{Ehrhart structure and finite reconstruction}\label{sec:structure}

We begin with the established structure that determines how many
counts are needed.
For $n\geq2$, put
\begin{equation*}
 d=(n-1)^2,\qquad s=(n-1)(n-2),\qquad K=s/2.
\end{equation*}
The integer $s$ is even. The Birkhoff--von Neumann theorem
\cite{Birkhoff46} implies that $\B_n$ is a lattice polytope; its dimension is $d$. Its Ehrhart series has the form
\begin{equation}\label{eq:series}
 \Ehr_{\B_n}(z)=\sum_{r\geq0}H_n(r)z^r
 =\frac{h_n^*(z)}{(1-z)^{d+1}},\qquad
 h_n^*(z)=\sum_{j=0}^s a_jz^j.
\end{equation}
When comparing different orders, we write $a_{n,j}=a_j$ to make
the dependence on $n$ explicit.
The numerator has nonnegative integer coefficients, degree $s$, and
\begin{equation}\label{eq:pal}
 a_0=a_s=1,\qquad a_j=a_{s-j}.
\end{equation}
Nonnegativity follows from Stanley's result for magic-labeling algebras
\cite[Proposition~4.5]{Stanley76}, applied to the complete bipartite graph
$K_{n,n}$; see also \cite[Theorem~2.1]{Stanley80} for the general
integral-polytope formulation. The degree and symmetry follow from
reciprocity and the codegree; the complete numerator statement is
\cite[Corollary~6.4]{BR}. Here the codegree is $n$.
The underlying interior translation
subtracts the matrix with all entries $1$ from a positive integral matrix. Equivalently,
the usual reciprocity relations specialize to
\begin{equation*}
 H_n(-n-r)=(-1)^dH_n(r),\qquad
 H_n(-1)=\cdots=H_n(-(n-1))=0.
\end{equation*}
Polynomiality and these reciprocity relations are already treated in
Stanley's work on magic labelings
\cite[Theorems~1.2 and~1.3, and Corollary~1.4]{Stanley73};
see also \cite[Proposition~4.6.2]{StanleyEC}.
These results provide the structural input for reconstruction.

\begin{proposition}\label{prop:reconstruct}
For $n\geq2$, the values $H_n(0),\ldots,H_n(K)$ determine the complete
Ehrhart series. The first half of its numerator is
\begin{equation}\label{eq:transform}
 a_k=\sum_{i=0}^k(-1)^{k-i}\binom{d+1}{k-i}H_n(i),
 \qquad 0\leq k\leq K.
\end{equation}
The remaining coefficients follow from~\eqref{eq:pal}, and
\begin{equation}\label{eq:binomial}
 H_n(r)=\sum_{j=0}^s a_j\binom{r+d-j}{d}
\end{equation}
is an identity of polynomials in $r$.
\end{proposition}
\begin{proof}
Extract the coefficient of $z^k$ after multiplying~\eqref{eq:series} by
$(1-z)^{d+1}$. Palindromicity supplies the other coefficients.
Expanding $(1-z)^{-d-1}$ gives~\eqref{eq:binomial} for every nonnegative
integer $r$, hence as a polynomial identity. For $j>r\geq0$, the upper
argument $r+d-j$ is a nonnegative integer smaller than $d$, so the
corresponding binomial coefficient is $0$.
\end{proof}

\subsection{Symmetric functions and the root filter}\label{sec:filter}
It remains to compute the values $H_n(r)$ needed for the reconstruction. For each row, the prescribed sum is encoded by a generating polynomial. Let
\begin{equation}\label{eq:hr}
 h_r(\mathbf{X})=\sum_{\substack{\alpha_1,\ldots,\alpha_n\geq0\\\alpha_1+\cdots+\alpha_n=r}}
 x_1^{\alpha_1}\cdots x_n^{\alpha_n}
\end{equation}
denote the complete homogeneous symmetric polynomial in $\mathbf{X}=(x_1,\ldots,x_n)$. This notation is distinct from the notation $h_n^*(z)$ for the Ehrhart numerator. Each factor of $h_r(\mathbf{X})^n$ corresponds to one row. Selecting a monomial from a factor selects the entries of that row. The exponents in the product record the column sums. Hence
\begin{equation}\label{eq:coefficient}
 H_n(r)=[x_1^r\cdots x_n^r]h_r(\mathbf{X})^n.
\end{equation}

For the root filter, let $n\geq3$, $r\geq0$, and let $m>r$ be an integer. Choose a prime $p$ satisfying
\begin{equation}\label{eq:prime}
 p>n,\qquad m\mid p-1.
\end{equation}
Let $\Omega=\{a\in\F_p:a^m=1\}$. This set has exactly $m$ elements. If $m=1$, it is the singleton $\{1\}$. In a finite-field identity, an integer on the right-hand side is understood after reduction modulo $p$.

\begin{theorem}\label{thm:filter}
Under~\eqref{eq:prime},
\begin{equation}\label{eq:filter}
 H_n(r)=m^{-n}\sum_{\mathbf{X}\in\Omega^n}
 \left(\prod_{i=1}^n x_i^{-r}\right)h_r(\mathbf{X})^n
 \quad\text{in }\F_p.
\end{equation}
\end{theorem}

\begin{proof}
For any integer $b$, character orthogonality gives
\[
 \frac1m\sum_{x\in\Omega}x^b=
 \begin{cases}1,&m\mid b,\\0,&m\nmid b.\end{cases}
\]
Consequently, a monomial $\mathbf{X}^{\boldsymbol{\alpha}}$ of $h_r(\mathbf{X})^n$ contributes to the right-hand side of~\eqref{eq:filter} exactly when $\alpha_i\equiv r\pmod m$ for every $i$. Since $0\leq r<m$, each nonnegative exponent can be written as $\alpha_i=r+m u_i$, with $u_i\geq0$. By homogeneity, $\sum_i\alpha_i=nr$, so $\sum_i u_i=0$. Since every $u_i\geq0$, this forces $u_i=0$ for all $i$. Hence the only surviving coefficient is the one in~\eqref{eq:coefficient}.
\end{proof}

\begin{remark}
The filter also follows from \cite[Theorem~4]{KP}. Apply the coefficient formula with every interpolation set equal to $\Omega$, every target exponent equal to $m-1$, and the polynomial
\[
 \left(\prod_{i=1}^n x_i^{m-1-r}\right)h_r(\mathbf{X})^n.
\]
This polynomial has degree $n(m-1)$, and its coefficient of $\prod_i x_i^{m-1}$ is $H_n(r)$.
\end{remark}

The same proof applies to any homogeneous polynomial of total degree $nr$ when the desired exponent vector is $(r,\ldots,r)$. In particular, $m=r+1$ is sufficient. More generally, to compute $H_n(0),\ldots,H_n(R)$, one may take $m=R+1$ and use the same primes and root set for every count. Then only $h_r$ and the weights $x_i^{-r}$ in the filter change with $r$. This also allows the root tables to be shared among the counts. The algorithm below instead chooses $m=r+1$ for each count. This choice minimizes the root set and, by~\eqref{eq:distinct-zero}, ensures that every multiset of distinct roots contributes $0$. A larger common root order need not preserve this vanishing for smaller values of $r$.

We use three reductions to speed up the computation.
\begin{enumerate}
\item Reduce the $m^n$ ordered tuples in the average by permutation and scaling symmetries, as in Subsection~\ref{sec:anchor}.
\item For a fixed multiset $M$, evaluate $h_r(M)$ efficiently. The usual recurrence for $\prod_{x\in M}(1-xu)^{-1}$ uses $O(nr)$ field operations up to degree $r$, or $O(nK)$ over the reconstruction range. Subsections~\ref{sec:local-ct} and~\ref{sec:dd} replace this calculation with local Todd coefficients. The number of terms in the resulting expression for $h_r(M)$ depends only on the multiplicities of the elements of $M$. In particular, it is independent of $r$ and $K$.

\item Combine the terms that belong to the same multiplicity class. Subsection~\ref{sec:agg-classes} uses elementary symmetric functions to sum over the choices of singleton roots. The finite-product identity and its applications in Subsection~\ref{sec:moments} extend this calculation to the polynomial statistics that occur at higher multiplicities.
\end{enumerate}

\subsection{Symmetry reduction of the root average}\label{sec:anchor}

For the rest of the method we choose $m=r+1$, the smallest root
order allowed by Theorem~\ref{thm:filter}. Thus $a^{-r}=a$ for
$a\in\Omega$. Let $\calM$ be the set of multisets of size $n$
from $\Omega$. For $M\in\calM$, let $\supp(M)$ be its set of
distinct elements and $\mu_a$ the multiplicity of $a$.
The symmetry of $h_r$ makes its value $h_r(M)$ independent of
the ordering of the entries. Permutation symmetry gives the
multinomial weights; a marking argument handles scaling.
For
$M\in\calM$, define
\begin{equation*}
 \mult(M)=\frac{n!}{\prod_a\mu_a!},\qquad
 \Theta(M)=\left(\prod_a a^{\mu_a}\right)h_r(M)^n,\qquad
 t(M)=\#\{a:\mu_a\geq2\}.
\end{equation*}
The integer $\mult(M)$ is the number of ordered tuples represented by $M$.
Homogeneity in~\eqref{eq:hr} gives
$h_r(\lambda M)=\lambda^r h_r(M)$.
Thus simultaneous scaling by $\lambda\in\Omega$
preserves multiplicities and $\Theta$, since
\begin{equation*}
 \Theta(\lambda M)=\lambda^{n+nr}\Theta(M)=\lambda^{nm}\Theta(M)=\Theta(M).
\end{equation*}

Multisets with no repeated root contribute $0$. Indeed, if such
an $M$ exists, then $n\leq m$ and
\[
 Q_M(u)=\frac{1-u^m}{\prod_{b\in M}(1-bu)}
       =\prod_{b\in\Omega\setminus M}(1-bu)
\]
is a polynomial of degree $m-n$. Since $r=m-1<m$,
\begin{equation}\label{eq:distinct-zero}
 h_r(M)=[u^r]\frac{Q_M(u)}{1-u^m}=[u^r]Q_M(u)=0.
\end{equation}
The last equality uses $m-n<r$. Thus only multisets with at least
one repeated root need to be considered. For a general root
order $m>r$, the same argument gives vanishing when $m-r<n$;
it need not give vanishing for the smaller counts evaluated
with a common larger root set.

The following identity gives the weights for marking a root. To mark a root means to select one eligible root value as a distinguished root. Dividing all roots by this value then normalizes the marked root to $1$.

\begin{lemma}\label{lem:eligible-marking}
Let $\mathcal A$ be a finite set with an action of $\Omega$, and let
$c:\mathcal A\to\F_p$ be invariant under this action. Suppose that each
$A\in\mathcal A$ is assigned a nonempty subset $E(A)\subseteq\Omega$
such that
\[
 E(\lambda A)=\lambda E(A),\qquad 1\leq |E(A)|<p.
\]
Then
\begin{equation*}
 \sum_{A\in\mathcal A}c(A)
 =m\sum_{\substack{A\in\mathcal A\\1\in E(A)}}
           \frac{c(A)}{|E(A)|}.
\end{equation*}
No freeness assumption on the action is required.
\end{lemma}
\begin{proof}
For each $A$, split its contribution equally among its eligible marks:
\[
 \sum_Ac(A)
 =\sum_{a\in\Omega}\ \sum_{\substack{A\in\mathcal A\\a\in E(A)}}
                  \frac{c(A)}{|E(A)|}.
\]
For each $a\in\Omega$, the action of $a^{-1}$ gives a bijection from the inner indexing set for mark $a$ to the inner indexing set for mark $1$. This bijection preserves the summand and the number of marks. Hence all $m$ inner sums are equal. This argument counts marked objects and never divides by the size of an orbit.
\end{proof}

\begin{theorem}\label{thm:anchor}
Let $\calM_1=\{M\in\calM:\mu_1\geq2\}$. Then
\begin{equation*}
 H_n(r)=m^{1-n}\sum_{M\in\calM_1}\frac{\mult(M)}{t(M)}\Theta(M)
 \quad\text{in }\F_p.
\end{equation*}
The formula is valid even when some scaling orbits have nontrivial
stabilizers.
\end{theorem}
\begin{proof}
Grouping ordered tuples in Theorem~\ref{thm:filter} gives
$m^{-n}\sum_M\mult(M)\Theta(M)$. As shown above, terms with
$t(M)=0$ vanish. On the remaining multisets, choose
$E(M)=\{a:\mu_a\geq2\}$ and apply
Lemma~\ref{lem:eligible-marking}. Here
$1\leq|E(M)|=t(M)\leq n<p$, so every denominator is invertible.
\end{proof}

More generally, for a scaling-invariant family $\mathcal F$ of
multisets, its contribution to the root average is
\begin{equation}\label{eq:eligible-family}
  m^{-n}\sum_{M\in\mathcal F}\mult(M)\Theta(M)
 =m^{1-n}\sum_{\substack{M\in\mathcal F\\1\in E(M)}}
       \frac{\mult(M)}{|E(M)|}\Theta(M),
\end{equation}
whenever the eligible roots satisfy the hypotheses of the lemma.

For the multiset enumeration, we mark roots
of largest multiplicity. Set
\[
 q(M)=\max_{a\in\supp M}\mu_a,\qquad
 \eta(M)=\#\{a:\mu_a=q(M)\}.
\]

\begin{corollary}\label{cor:max-anchor}
Let $\calM_{\max}=\{M\in\calM:\mu_1=q(M)\geq2\}$. Then
\begin{equation}\label{eq:anchor-max}
 H_n(r)=m^{1-n}\sum_{M\in\calM_{\max}}
                  \frac{\mult(M)}{\eta(M)}\Theta(M).
\end{equation}
Restricting this sum to a scaling-invariant family gives that
family's contribution.
\end{corollary}
\begin{proof}
Take $E(M)=\{a:\mu_a=q(M)\}$. Multiplication of roots by $\lambda$
preserves both maximality and the number of ties. The set is nonempty,
and $\eta(M)\leq n<p$. Apply~\eqref{eq:eligible-family}, omitting the
all-distinct multisets, whose contributions vanish.
\end{proof}

In particular, after fixing the multiplicity $q\ge2$ of the marked root at $1$, the enumeration need only consider other multiplicities at most $q$.
 In its summand, the divisor is $\eta(M)$, not $t(M)$. These two quantities may differ, even though they arise from the same marking identity. It remains to evaluate $h_r(M)$ in these weighted sums. The next two subsections express $h_r(M)$ as a sum of local constant terms and compute these terms using generalized Todd coefficients. Together with~\eqref{eq:anchor-max}, this gives every required count modulo an admissible prime.

\subsection{Reduction to local constant terms}\label{sec:local-ct}

Fix a multiset $M\in\calM$. We retain $m=r+1$ and the prime
conditions~\eqref{eq:prime}. Products and sums over a multiset
include each entry with its multiplicity.

Over a field $\mathbb K$, for a Laurent series
$L(u)=\sum_{j\geq j_0}L_j u^j$, write
$\CT\limits_u L(u)=L_0$.  Rational functions at $u=0$ are expanded in their
Laurent series there; in particular, $(1-cu)^{-1}$ has its geometric
expansion.  Since a series containing only positive powers has constant term $0$,
\begin{equation*}
 \CT\limits_u L(u)A(u)=\CT\limits_u\bigl(L(u)-u^\ell\bigr)A(u),
 \qquad A(u)\in\mathbb K[[u]],\quad \ell\geq1.
\end{equation*}
The generating function of the complete homogeneous symmetric
polynomials therefore gives
\begin{equation*}
 h_r(M)=\CT\limits_u E_M(u),\qquad
 E_M(u)=\frac{u^{-r}-u}{\prod_{b\in M}(1-bu)},
\end{equation*}

For local contributions, we use the following notation.
If $\alpha\ne0$ and the partial-fraction
block of a rational function $E(u)$ at $\alpha$ is
\[
 \frac{A_\alpha(u)}{(1-u/\alpha)^\ell},\qquad
 \deg A_\alpha<\ell,
\]
then $\CT\limits_{u=\alpha}E(u):=A_\alpha(0)$; the contribution is $0$ if
$E$ is regular at $\alpha$.  This notation refers to a whole
partial-fraction block, rather than the constant coefficient of
$E(\alpha+u)$.  The partial-fraction approach of Xin~\cite[Section~4]{Xin04}
allows the desired constant term to be recovered from these blocks.

Here $E_M$ is proper as a rational function of $u$: its behavior at
infinity is $O(u^{1-n})$.  Its principal part at $0$ has only negative
powers and contributes nothing to $\CT\limits_u E_M$.  At $u=a^{-1}$,
the numerator $u^{-r}(1-u^m)$ has a simple zero because $a^m=1$
and $p\nmid m$.  A root of multiplicity $\mu_a$ consequently produces
a pole of order $\mu_a-1$, if this number is positive.  Thus
\begin{equation}\label{eq:local}
 h_r(M)=\sum_{\substack{a\in\supp(M)\\\mu_a\geq2}}
              \CT\limits_{u=a^{-1}}E_M(u).
\end{equation}
This also proves the vanishing for all-distinct multisets directly
from the constant-term representation.

The choice $m=r+1$ ensures both the smallest root set and a proper
rational function after cancellation. With a common root order
$m>r$, the analogous subtraction gives numerator
$u^{-r}-u^{m-r}$. The resulting rational function is
$O(u^{m-r-n})$ at infinity.
When $m-r\geq n$, a polynomial part may contribute to the
constant term. Thus the common-root filter remains valid, but
the repeated-root reduction above uses the stated choice of $m$.

\subsection{Computation of the local Todd coefficients}\label{sec:dd}

Let $a$ have multiplicity $\mu_a\geq2$ in $M$, and write its
partial-fraction block in $E_M(u)$ as
$A_{a^{-1}}(u)/(1-au)^{\mu_a-1}$, with
$\deg A_{a^{-1}}<\mu_a-1$.
We compute $A_{a^{-1}}(0)=\CT\limits_{u=a^{-1}}E_M(u)$.
Put $M'=M-a^*$, where $a^*$ denotes all occurrences of $a$;
thus $M'$ has $n-\mu_a$ entries and contains no $a$. Also write
$M'/a$ for the multiset obtained by dividing each entry by $a$.

In characteristic $0$, the change of variables of Liu and
Xin~\cite[Proposition~2.4]{LX26} reads
\begin{equation}\label{eq:lx-change}
 \CT\limits_{u=\alpha}E(u)=-\CT\limits_s\,sE(\alpha e^s).
\end{equation}
We apply it before reducing the necessary coefficients modulo $p$.
All exponential, logarithmic, and Todd series below are used only to
the stated finite order.  They do not denote infinite exponential or
logarithmic series over $\F_p$.

Put
\begin{equation*}
 T(s)=\frac{s}{e^s-1},\qquad
 R_c(s)=\frac{1-c}{1-ce^s}
       =\frac{1}{1-y_c(e^s-1)},\qquad y_c=\frac{c}{1-c}
       \quad(c\ne1).
\end{equation*}
These are the two normalized (power series) factors used by Xin, Zhang, and
Zhang~\cite[Section~3]{XZZ25} to define generalized Todd polynomials.

\begin{proposition}\label{thm:todd-local}
Let $a\in\supp(M)$ have multiplicity $\mu_a\geq2$.
Define
\begin{align}
 C_a&=\frac{(-1)^{\mu_a} m}{a}
       \prod_{b\in M'}(1-b/a)^{-1},\label{eq:todd-prefactor}\\
 F_a(s)&=e^{-rs}\frac{T(s)^{\mu_a}}{T(ms)}
       \prod_{b\in M'}R_{b/a}(s).\label{eq:todd-normalized}
\end{align}
Then
\begin{equation}\label{eq:todd-local}
 \CT\limits_{u=a^{-1}}E_M(u)=C_a[s^{\mu_a-2}]F_a(s),\qquad
 h_r(M)=\sum_{\substack{a\in\supp(M)\\\mu_a\geq2}}
                    C_a[s^{\mu_a-2}]F_a(s).
\end{equation}
The required truncations $F_a(s)\bmod s^{\mu_a-1}$ are well defined
modulo every prime satisfying~\eqref{eq:prime}.
\end{proposition}
\begin{proof}
Since $a^r=a^{-1}$, the summand in~\eqref{eq:local}, after applying
\eqref{eq:lx-change}, is
\[
 -\frac1a\CT\limits_s
 \frac{s e^{-rs}(1-e^{ms})}{(1-e^s)^{\mu_a}}
 \prod_{b\in M'}(1-(b/a)e^s)^{-1}.
\]
Using $1-e^s=-s/T(s)$ and $1-e^{ms}=-ms/T(ms)$ transforms the
Laurent series inside this constant term, including the factor $-1/a$,
into $C_a s^{2-\mu_a}F_a(s)$.  This proves~\eqref{eq:todd-local}.

Only coefficients of degree at most $\mu_a-2\leq n-2$ in the normalized
factors are needed.  They can be constructed by finite polynomial
operations from the coefficients of $e^s$ through degree $n-1$,
division by series with constant coefficient $1$, and division by
integers at most $n-2$ for logarithms and exponentials.  Thus their
numerical denominators are products of integers smaller than $n$;
the other denominators are powers of the nonzero elements $1-b/a$.
The condition $p>n$ makes every such operation valid, while
$m\mid p-1$ ensures $m\ne0$ in $\F_p$.  The characteristic-$0$
identity therefore reduces to the asserted finite-field identity.
\end{proof}

In the parametrization of~\cite[Equation~(5)]{XZZ25}, the generalized
Todd series~\eqref{eq:todd-normalized} has exponential parameter $-r$.
Before normalization, the denominator factor $(1-e^s)^{\mu_a}$
gives a parameter multiset consisting of $\mu_a$ copies of $1$,
and the numerator factor $1-e^{ms}$ gives the parameter multiset
$\{m\}$. Each occurrence $b\in M'$ supplies one group with
parameter $y_{b/a}$. All other numerator groups are empty. Equal parameters may be retained as separate factors. The factor $s$ in~\eqref{eq:lx-change}, together with the zero $1-e^{ms}$, lowers the required coefficient order to $\mu_a-2$. Hence the local computation is precisely a generalized Todd coefficient computation. It uses the normalization and
logarithm--exponential steps underlying Algorithm CTGTodd
\cite[Section~4]{XZZ25}, with this coefficient order.

The specialization used in Algorithm~1 is as follows. Write
\[
 \log T(s)=\sum_{j\geq1}\beta_j s^j,
 \qquad
 \beta_1=-\frac12,\quad
 \beta_{2v}=-\frac{B_{2v}}{(2v)(2v)!}\ (v\geq1),\quad
 \beta_{2v+1}=0\ (v\geq1),
\]
where the Bernoulli numbers are defined by
$T(s)=\sum_{j\geq0}B_j s^j/j!$.
Equations~(6) and (7) of~\cite{XZZ25} give
\[
 \log R_c(s)=\sum_{j\geq1}C_j(y_c)s^j,
 \qquad
 C_j(y)=\frac1{j!}\sum_{v=1}^j(v-1)!\,
                   \genfrac{\{} {\}}{0pt}{}{j}{v}y^v,
\]
where $\genfrac{\{} {\}}{0pt}{}{j}{v}$ is a Stirling number of the second kind.
For a multiset $N$ of roots different from $1$, define the additive
statistics
\[
 \mathcal S_j(N)=\sum_{c\in N}C_j(y_c),\qquad 1\leq j\leq n-2.
\]
Every occurrence is included in this sum. It follows that, for
$1\leq j\leq \mu_a-2$,
\begin{equation}\label{eq:todd-log}
 \ell_{a,j}:=[s^j]\log F_a(s)
 =-r\,\mathbf1_{j=1}+(\mu_a-m^j)\beta_j
  +\mathcal S_j(M'/a).
\end{equation}
Given these coefficients, apply~\cite[Lemma~4]{XZZ25}: set $f_{a,0}=1$ and compute
\begin{equation}\label{eq:todd-exp-recurrence}
 f_{a,j}=\frac1j\sum_{i=1}^j i\ell_{a,i}f_{a,j-i},
 \qquad 1\leq j\leq \mu_a-2.
\end{equation}
Indeed, differentiating $F_a=\exp(\log F_a)$ gives
$F_a'=(\log F_a)'F_a$; comparison of the coefficient of $s^{j-1}$
gives exactly~\eqref{eq:todd-exp-recurrence}. Thus
$f_{a,j}=[s^j]F_a(s)$. For each repeated root, compute $C_a$ and
the statistics $\mathcal S_j(M'/a)$ for $j\leq \mu_a-2$, form the
logarithmic coefficients by~\eqref{eq:todd-log}, and apply the
recurrence through degree $\mu_a-2$. Summing $C_af_{a,\mu_a-2}$ over the
repeated roots completes the general evaluation of $h_r(M)$.

All ratios $b/a$ belong to $\Omega$.  Hence the arrays
$C_j(y_c)$ for $c\in\Omega\setminus\{1\}$ and $j\leq n-2$ can be
computed once for each pair $(m,p)$ and reused for every multiset.
The same is true of the coefficients $\beta_j$ and the powers $m^j$.
In particular, the formal parameters of the general Todd polynomial
have already been specialized to field elements: the subsequent
logarithm and exponential calculations are univariate.

The first cases give useful explicit formulas. Put
\[
 P_a(M)=\prod_{b\in M'}(1-b/a)^{-1}.
\]
Since $C_1(y)=y$ and $C_2(y)=(y+y^2)/2$, the statistics needed
at these multiplicities are explicitly
\[
 \mathcal S_1(M'/a)=\sum_{b\in M'}\frac{b}{a-b},\qquad
 \mathcal S_2(M'/a)=\frac12\sum_{b\in M'}\frac{ab}{(a-b)^2}.
\]
The local contribution is
\begin{equation}\label{eq:todd-small}
 A_{a^{-1}}(0)=\frac{m}{a}P_a(M)
 \begin{cases}
  1, & \mu_a=2,\\[3pt]
  \dfrac{m+1}{2}-\mathcal S_1(M'/a), & \mu_a=3,\\[6pt]
  \dfrac12\left(\mathcal S_1(M'/a)-\dfrac{m+2}{2}\right)^2
  +\mathcal S_2(M'/a)+\dfrac{m^2-4}{24}, & \mu_a=4.
 \end{cases}
\end{equation}
These expressions follow from $f_{a,0}=1$,
$f_{a,1}=\ell_{a,1}$, and
$f_{a,2}=\ell_{a,2}+\ell_{a,1}^2/2$.
Already at multiplicity $4$, two additive statistics and a
quadratic expression are required. Larger multiplicities are
handled uniformly by the recurrence. The low-multiplicity
formulas will also allow whole classes to be summed by moments
in Subsections~\ref{sec:agg-classes} and~\ref{sec:moments}.

The statistics permit reuse between nearby multisets. Fix $a$,
and suppose $M_1=M_2-\{x\}+\{y\}$ with $x,y\ne a$, where one
occurrence of $x$ is removed and one of $y$ is inserted. The
multiplicity of $a$ is unchanged. Writing $M_i'=M_i-a^*$ gives
\begin{align*}
 \mathcal S_j(M_1'/a)
  &=\mathcal S_j(M_2'/a)-C_j(y_{x/a})+C_j(y_{y/a}),\\
 P_a(M_1)&=P_a(M_2)\frac{1-x/a}{1-y/a}.
\end{align*}
Thus all required statistics and the prefactor can be updated
without rescanning the common entries. For a fixed multiplicity
pattern, the implementation uses the same identities along
shared partial root assignments, as described in
Subsection~\ref{sec:crt}.

Finally, index $\Omega$ as $\{1,\omega,\ldots,\omega^{m-1}\}$.
Products and quotients of roots then correspond to addition
and subtraction of indices modulo $m$; for instance,
$b/a=\omega^{j-i}$ if $a=\omega^i$ and $b=\omega^j$.
After the root and ratio tables are constructed, these operations
use integer index arithmetic and table lookup. Operations on
general field elements, including $1-b/a$ and the sums in
the recurrence, are still counted as field operations.

\subsection{Summing multiplicity classes}\label{sec:agg-classes}

Although the local Todd coefficients have small degrees, the
number of multisets can still make their individual evaluation
costly. For classes with low repeated multiplicities, most of
these choices come from the singleton roots. We first explain
how elementary symmetric functions combine these choices for
the pattern $(2^t)$, and then describe the additional polynomial
statistics needed at higher multiplicities. The finite-product
evaluation of these statistics is proved and illustrated in
Subsection~\ref{sec:moments}. We use primes $p>2n$ for the
class sums, so that all required factorials are invertible.
The root filter and local Todd formula themselves remain valid
under~\eqref{eq:prime}.

Let $R\subseteq\Omega$ consist of $t$ distinct double roots,
where $1\leq t\leq\lfloor n/2\rfloor$, and put $k=n-2t$.
The remaining roots form a set $S\subseteq\Omega\setminus R$
of size $k$. Thus $M=R^2\cup S$, where $R^2$ includes each
element of $R$ twice. Even for this fixed $R$, there are
$\binom{m-t}{k}$ possible singleton sets. Infeasible choices
$k>m-t$ are omitted. Once the inverse root differences and the
factors $q_a$ below are available, recomputing the products
separately for each $S$ takes
\[
 O\!\left(\bigl(t(k+1)+\log(n+1)\bigr)\binom{m-t}{k}\right)
\]
field operations: there are $t$ local products, and the resulting
value must be raised to the $n$th power. We will sum these
contributions using elementary symmetric functions.

At a double root $a$, the local Todd coefficient is $f_{a,0}=1$. Since the multiset $M-a^*$ has
$n-2$ entries, its contribution is
\[
 \frac ma\prod_{b\in M-a^*}(1-b/a)^{-1}
 =m a^{n-3}
   \prod_{b\in R\setminus\{a\}}(a-b)^{-2}
   \prod_{s\in S}(a-s)^{-1}.
\]
Define
\[
 q_a=a^{n-3}\prod_{b\in R\setminus\{a\}}(a-b)^{-2}.
\]
Summing the local contributions gives
\begin{equation}\label{eq:agg-double-local}
 m^{-1}h_r(R^2,S)=\sum_{a\in R}q_a\prod_{s\in S}(a-s)^{-1}.
\end{equation}
On raising this expression to the $n$th power, each term of the
multinomial expansion is specified by an exponent vector
$\mathbf j=(j_a)_{a\in R}\in\N^R$ with
$|\mathbf j|:=\sum_{a\in R}j_a=n$. Consequently,
\[
 h_r(R^2,S)^n
 =m^n\sum_{|\mathbf j|=n}
   \binom n{(j_a)_{a\in R}}
   \left(\prod_{a\in R}q_a^{j_a}\right)
   \prod_{s\in S}\prod_{a\in R}(a-s)^{-j_a}.
\]
For this pattern, the multinomial weight of the multiset is
$\mult(M)=n!/2^t$, every repeated root has maximal multiplicity,
and $\eta(M)=t$. Its root-product factor in $\Theta(M)$ is
$\prod_{a\in R}a^2\prod_{s\in S}s$.
After fixing the mark at $1\in R$, Corollary~\ref{cor:max-anchor}
therefore gives, for this fixed repeated-root set, the contribution
\[
 \frac{n!m}{2^t t}\left(\prod_{a\in R}a^2\right)
 \sum_{|\mathbf j|=n}\binom n{(j_a)_{a\in R}}
   \left(\prod_{a\in R}q_a^{j_a}\right)P_R(\mathbf j),
\]
where
\[
 P_R(\mathbf j)=\sum_{\substack{S\subseteq\Omega\setminus R\\|S|=k}}
                  \prod_{s\in S}w_s(\mathbf j),
 \qquad
 w_s(\mathbf j)=s\prod_{a\in R}(a-s)^{-j_a}.
\]
The factor $m$ comes from $m^{1-n}m^n$. Each singleton set in
the inner sum is chosen without an ordering; no additional
factor $k!$ is needed.
Indeed, $P_R(\mathbf j)$ is the $k$th elementary symmetric
function in the values $w_s(\mathbf j)$, indexed by
$s\in\Omega\setminus R$. Hence
\begin{equation}\label{eq:agg-double-subset}
 P_R(\mathbf j)=[v^k]\prod_{s\in\Omega\setminus R}
                           (1+v w_s(\mathbf j)).
\end{equation}
The standard evaluation of this elementary symmetric function
uses $O((m-t)(k+1))$ field operations once its arguments are known.
There are $\binom{n+t-1}{t-1}$ exponent vectors. For $t\leq4$,
preparing the weights and evaluating the elementary symmetric
function for all these vectors takes
\[
 O\!\left(\binom{n+t-1}{t-1}(m+1)(n+1)\right)
\]
field operations for this fixed $R$; the powers needed in the
weights may be prepared once and reused. For fixed $n$ and $t$
with $k\geq2$, the dependence on $m$ decreases from degree $k$
in the separate evaluation to degree $1$. When $k=1$, both
calculations are linear in $m$; when $k=0$, $P_R(\mathbf j)=1$
and no singleton calculation is needed. The binomial factor
counting exponent vectors must still be taken into account
when $n$ and $t$ vary.

For example, let $n=6$, $t=2$, and $R=\{1,a\}$, where $a\ne1$.
There are $k=2$ singleton roots and seven exponent vectors
$\mathbf j=(j,6-j)$, with $0\leq j\leq6$. For each such vector,
\[
 P_{\{1,a\}}(j,6-j)
 =[v^2]\prod_{s\in\Omega\setminus\{1,a\}}
     \left(1+\frac{v s}{(1-s)^j(a-s)^{6-j}}\right).
\]
This is the elementary symmetric function of degree $2$ in the
displayed singleton weights. Combining these values with the
multinomial and root factors gives the entire $(2,2)$ class:
\[
 90m\sum_{a\in\Omega\setminus\{1\}}\frac{a^2}{(1-a)^{12}}
      \sum_{j=0}^{6}\binom6j a^{3(6-j)}
                P_{\{1,a\}}(j,6-j).
\]
Here $90=6!/(2^2\cdot2)$, and the common denominator
$(1-a)^{12}$ comes from $q_1^j q_a^{6-j}$.

In general, summing over the sets $R\subseteq\Omega$ with $\lvert R\rvert=t$ and $1\in R$ gives
\begin{equation}\label{eq:agg-double-total}
 \frac{n!m}{2^t t}
 \sum_{\substack{R\subseteq\Omega\\|R|=t,\ 1\in R}}
 \left(\prod_{a\in R}a^2\right)
 \sum_{|\mathbf j|=n}
       \binom n{(j_a)_{a\in R}}
       \left(\prod_{a\in R}q_a^{j_a}\right)P_R(\mathbf j).
\end{equation}
There are $\binom{m-1}{t-1}$ choices of $R$. Multiplying the
fixed-$R$ bound above by this number gives the complete-class
bound in Subsection~\ref{sec:computation}.

For $t=4$, the expansion is reduced further by marking a root
where $j_a$ is maximal. Each expanded term is invariant under
scaling by $\lambda$: $q_a\prod_{s\in S}(a-s)^{-1}$ has degree $-1$ in $\lambda$, and
its $n$th power cancels the degree $n$ in $\lambda$ of the root product.
Put $\tau(\mathbf j)=\#\{a\in R:j_a=\max_b j_b\}$.
In~\eqref{eq:agg-double-total}, omit the outer factor $1/t$,
restrict the inner sum to $j_1\geq j_a$ for every $a\in R$,
and divide each term by $\tau(\mathbf j)$. The marking lemma
proves this identity even when scaling has stabilizers. At
$n=12$, it retains $127$ of the $455$ exponent vectors.

For higher multiplicities, the singleton sum also contains
polynomials in additive statistics. At a triple root marked at
$1$, for example, the required statistic is
$\sum_{s\in S}(1-s)^{-1}$. A quadruple root also requires
$\sum_{s\in S}(1-s)^{-2}$. These sums arise because
\eqref{eq:todd-log} writes each local logarithmic coefficient
as a constant plus a sum of contributions from individual roots.
The exponential recurrence makes each Todd coefficient a
polynomial in those sums.

After expanding $h_r(M)^n$ and collecting the multiplicative
factors, the required calculation therefore has the form
\begin{equation}\label{eq:weighted-subset-model}
 \sum_{\substack{S\subseteq U\\|S|=k}}
       \left(\prod_{s\in S}w_s\right)
       Q\left(\sum_{s\in S}\mathbf v_s\right).
\end{equation}
Here $U$ is the set of available singleton roots, $k$ is the
number to be chosen, and $w_s$ contains the multiplicative
factors associated with the root $s$. The vector $\mathbf v_s$, with $\kappa$ components and usually of the form
$((a-s)^{-1},\dots, (a-s)^{-\kappa}) $,
records its contributions to the additive statistics; $Q$ is
the polynomial in these statistics supplied by the Todd formula
and the selected term of the multinomial expansion. For $(2^t)$,
there are no additive statistics and $Q=1$, so this expression
is the elementary symmetric function in~\eqref{eq:agg-double-subset}.

Lemma~\ref{lem:exponential-moments} evaluates
\eqref{eq:weighted-subset-model} by processing the elements of
$U$ one at a time and retaining only the required subset sizes
and polynomial coefficients. For fixed $k$ and a fixed set of
coefficient indices, its cost is linear in $|U|$. Thus it avoids
enumerating the $\binom{|U|}{k}$ subsets. The next subsection
proves this identity and gives the explicit correspondence with
the model for triple and quadruple roots.

\subsection{Finite-product evaluation and applications}\label{sec:moments}

The subset size is recorded by a variable $v$. Additional
variables $\mathbf z=(z_1,\ldots,z_{\kappa})$ record powers of the
additive statistics. The required contribution from subsets of size $k$
can then be extracted using exponential generating functions.

\begin{lemma}\label{lem:exponential-moments}
Let $k\in\N$, let $U$ be a finite set, let $w_s\in\F_p$, and let
$\mathbf v_s=(v_{s,1},\ldots,v_{s,\kappa})\in\F_p^{\kappa}$ for $s\in U$.
Let $\mathcal I\subseteq\N^{\kappa}$ be a nonempty finite downward-closed set:
if $\boldsymbol\gamma\in\mathcal I$ and
$\mathbf0\leq\boldsymbol\delta\leq\boldsymbol\gamma$
componentwise, then $\boldsymbol\delta\in\mathcal I$.
Assume that every component of every index in $\mathcal I$ is
smaller than $p$. Write $\mathbf z=(z_1,\ldots,z_{\kappa})$,
$\boldsymbol\gamma!=\prod_i\gamma_i!$, and
$\mathbf u^{\boldsymbol\gamma}=\prod_i u_i^{\gamma_i}$
for any vector $\mathbf u$.

For $0\leq\ell\leq k$ and
$\boldsymbol\gamma\in\mathcal I$,
\begin{align}
 E_{\ell,\boldsymbol\gamma}
 &=\CT\limits_{v,z_1,\ldots,z_{\kappa}}
   v^{-\ell}\mathbf z^{-\boldsymbol\gamma}
   \prod_{s\in U}\left(1+v w_s
          \exp\left(\sum_{i=1}^{\kappa} z_i v_{s,i}\right)\right)
       \label{eq:agg-exp-ct}\\
 &=\frac1{\boldsymbol\gamma!}
   \sum_{\substack{S\subseteq U\\|S|=\ell}}
       \left(\prod_{s\in S}w_s\right)
       \prod_{i=1}^{\kappa}\left(\sum_{s\in S}v_{s,i}\right)^{\gamma_i}.
       \notag
\end{align}

Expand each multivariate exponential as
$\prod_i\exp(z_i v_{s,i})$, truncate to the required indices,
and then reduce modulo $p$; the coefficients involve only the
factorials allowed by the hypothesis.
Initialize $E_{0,\mathbf0}=1$ and all other entries $0$.
Inserting a new element $s$ gives
\begin{equation}\label{eq:agg-exp-update}
 E_{\ell,\boldsymbol\gamma}\leftarrow E_{\ell,\boldsymbol\gamma}
 +w_s\sum_{\mathbf0\leq\boldsymbol\delta\leq\boldsymbol\gamma}
   \frac{\mathbf v_s^{\boldsymbol\gamma-\boldsymbol\delta}}
        {(\boldsymbol\gamma-\boldsymbol\delta)!}
   E_{\ell-1,\boldsymbol\delta}.
\end{equation}
Update $\ell=k,k-1,\ldots,1$, retaining only indices
$\boldsymbol\gamma\in\mathcal I$.
\end{lemma}
\begin{proof}
In the product in~\eqref{eq:agg-exp-ct}, choosing the $v$-term
in the factors indexed by $S$ gives
\[
 v^{|S|}\left(\prod_{s\in S}w_s\right)
 \exp\left(\sum_{i=1}^{\kappa} z_i\sum_{s\in S}v_{s,i}\right).
\]
Extracting $[v^\ell\mathbf z^{\boldsymbol\gamma}]$ proves the
displayed identity. Inserting a new factor leaves the subsets
not containing $s$ unchanged. Its other term multiplies the
previous product by $v w_s\exp(\sum_i z_i v_{s,i})$.
Coefficient convolution gives~\eqref{eq:agg-exp-update}.
Descending order in $\ell$ ensures that the values at cardinality
$\ell-1$ have not yet included $s$. The hypotheses on
$\mathcal I$ ensure that every required coefficient and factorial
inverse is available.
\end{proof}

For a polynomial in $\kappa$ variables,
$Q(\mathbf z)=\sum_{\boldsymbol\gamma\in\mathcal I}
q_{\boldsymbol\gamma}\mathbf z^{\boldsymbol\gamma}$, the lemma gives
\begin{equation}\label{eq:agg-moment-contraction}
 \sum_{\substack{S\subseteq U\\|S|=k}}
   \left(\prod_{s\in S}w_s\right)
   Q\left(\sum_{s\in S}\mathbf v_s\right)
 =\sum_{\boldsymbol\gamma\in\mathcal I}
   q_{\boldsymbol\gamma}\boldsymbol\gamma! E_{k,\boldsymbol\gamma}.
\end{equation}
For $\kappa=0$, there are no moment variables, and
$E_{\ell,\mathbf0}$ is the elementary symmetric function of
degree $\ell$ in the weights $w_s$. This recovers the calculation
for the $(2^t)$ class.

For fixed $\kappa$, applying the recurrence to all indices in
$\mathcal I$ takes
\begin{equation}\label{eq:moment-cost}
 O\left(|U|(k+1)\sum_{\boldsymbol\gamma\in\mathcal I}
                              \prod_{i=1}^{\kappa}(\gamma_i+1)\right)
\end{equation}
field operations and requires $O((k+1)|\mathcal I|)$ working
storage. Indeed, for each $\boldsymbol\gamma$ there are
$\prod_i(\gamma_i+1)$ indices
$\boldsymbol\delta\leq\boldsymbol\gamma$ in the update.
The monomials in each new factor can be prepared within the
same bound. The final sum in~\eqref{eq:agg-moment-contraction}
uses $O(|\mathcal I|)$ further operations. In the applications
below, $\kappa$ is $1$ or $2$.

We now apply the lemma to patterns containing a triple or
quadruple root. For a repeated root $a$, put
\[
 \alpha=\frac{m+2n-5}{2},\qquad
 \beta=\frac{m^2+3m(n-3)+3n^2-18n+26}{6},\qquad
 \sigma_{a,j}=\sum_{b\in M-a^*}\frac1{(a-b)^j}.
\]
The coefficients required for multiplicities $2$, $3$, and $4$
are respectively
\begin{align}
 f_{a,0}&=1 &&(\mu_a=2),\label{eq:agg-f0}\\
 f_{a,1}&=-\alpha+a\sigma_{a,1} &&(\mu_a=3),\label{eq:agg-f1}\\
 f_{a,2}&=\beta-\alpha a\sigma_{a,1}
       +\frac{a^2}{2}(\sigma_{a,1}^2+\sigma_{a,2})
                                      &&(\mu_a=4).\label{eq:agg-f2}
\end{align}
Indeed, with $M'=M-a^*$,
\[
 \mathcal S_1(M'/a)
 =\sum_{b\in M'}\frac{b}{a-b}=a\sigma_{a,1}-(n-\mu_a),\qquad
 \mathcal S_2(M'/a)=\frac{a^2\sigma_{a,2}-a\sigma_{a,1}}2.
\]
Substitution in~\eqref{eq:todd-small} gives these formulas.
In particular,~\eqref{eq:agg-f0} is the local coefficient used
in~\eqref{eq:agg-double-local}.

Consider first the pattern $(3)$, with the triple root marked
at $1$. Put $k=n-3$ and $\xi_s=(1-s)^{-1}$. The local Todd
formula gives
\[
 m^{-1}h_r(1^3,S)
 =\left(\alpha-\sum_{s\in S}\xi_s\right)\prod_{s\in S}\xi_s.
\]
Consequently, the class contribution is $n!m/6$ times
\eqref{eq:weighted-subset-model}, with the following choices:
\[
 U=\Omega\setminus\{1\},\qquad k=n-3,\qquad \kappa=1,
 \qquad w_s=s \xi_s^n,\quad v_s=\xi_s,\quad Q(X)=(\alpha-X)^n.
\]
In this case the lemma computes the coefficients
\[
 E_{k,\ell}=[v^k z^\ell]
       \prod_{s\in\Omega\setminus\{1\}}
           (1+v s \xi_s^n e^{z \xi_s}),\qquad 0\leq\ell\leq n.
\]
Expanding $(\alpha-X)^n$ and applying
\eqref{eq:agg-moment-contraction} gives the complete contribution
\[
 \frac{n!m}{6}\sum_{\ell=0}^{n}
       (-1)^\ell\binom n\ell\alpha^{n-\ell}\ell!E_{n-3,\ell}.
\]
Thus the weighted sums of $1$, $\sum \xi_s$, $(\sum \xi_s)^2$,
and the higher powers through degree $n$ are obtained from
one coefficient array. The factor $\ell!$ converts its exponential
coefficients to the ordinary powers in $Q$.

The same calculation includes double roots alongside the
triple root. Consider the pattern $(3,2^\nu)$ with
$0\leq\nu\leq2$.
Scale the triple root to $1$ and write $R=\{1\}\cup\calD$,
where $\calD$ consists of $\nu$ double roots.
Put $k=n-3-2\nu$ and
\[
 L=\prod_{b\in\calD}(1-b)^{-2},\qquad
 c=\alpha-2\sum_{b\in\calD}(1-b)^{-1},\qquad
 q_b=b^{n-3}(b-1)^{-3}
       \prod_{c'\in\calD\setminus\{b\}}(b-c')^{-2}.
\]
Equations~\eqref{eq:agg-f0} and~\eqref{eq:agg-f1} give
\[
 m^{-1}h_r(1^3,\calD^2,S)
 =L\left(c-\sum_{s\in S}\frac1{1-s}\right)
      \prod_{s\in S}(1-s)^{-1}
   +\sum_{b\in\calD}q_b\prod_{s\in S}(b-s)^{-1}.
\]
For $\mathbf j=(j_1,(j_b)_{b\in\calD})$ with $|\mathbf j|=n$, define
\[
 w_s(\mathbf j)=s(1-s)^{-j_1}\prod_{b\in\calD}(b-s)^{-j_b},
 \qquad
 \Phi_{\calD}(\mathbf j)
 =\sum_{\substack{S\subseteq\Omega\setminus R\\|S|=k}}
   \left(\prod_{s\in S}w_s(\mathbf j)\right)
   \left(c-\sum_{s\in S}(1-s)^{-1}\right)^{j_1}.
\]
For fixed $\calD$ and $\mathbf j$, this is
\eqref{eq:weighted-subset-model} with $U=\Omega\setminus R$,
the displayed weights $w_s(\mathbf j)$, one statistic
$v_s=(1-s)^{-1}$, and $Q(X)=(c-X)^{j_1}$. The moment dimension
is $\kappa=1$, regardless of the number $\nu$ of double roots.
With these weights and statistic, Lemma~\ref{lem:exponential-moments}
and the binomial expansion give
\[
 \Phi_{\calD}(\mathbf j)
 =\sum_{\ell=0}^{j_1}(-1)^\ell\binom{j_1}{\ell}
       c^{j_1-\ell}\ell!E_{k,\ell}.
\]
For example, the mixed pattern $(3,2)$ has $\calD=\{b\}$,
$k=n-5$, and $j_b=n-j_1$. Its parameters are explicitly
\[
 U=\Omega\setminus\{1,b\},\qquad
 c=\alpha-\frac{2}{1-b},\qquad
 w_s(\mathbf j)=\frac{s}{(1-s)^{j_1}(b-s)^{n-j_1}}.
\]
For each $j_1$, the same one-variable calculation sums all
$\binom{m-2}{n-5}$ choices of singleton roots.
The full contribution of $(3,2^\nu)$ is
\[
 \frac{n!m}{6\,2^\nu}
 \sum_{\substack{\calD\subseteq\Omega\setminus\{1\}\\|\calD|=\nu}}
 \left(\prod_{b\in\calD}b^2\right)
 \sum_{|\mathbf j|=n}
   \binom n{j_1,(j_b)_{b\in\calD}}
   L^{j_1}\left(\prod_{b\in\calD}q_b^{j_b}\right)
   \Phi_{\calD}(\mathbf j).
\]
The unique triple root is the eligible mark, so there is no
division by the number of double roots. The required moment
degrees are at most $n$.

For the pattern $(4)$, mark the unique repeated root at $1$,
put $k=n-4$, and let $\xi_s=(1-s)^{-1}$. Define
\[
 Q(X,Y)=\left(\beta-\alpha X+\frac{X^2+Y}{2}\right)^n.
\]
Equation~\eqref{eq:agg-f2} gives the complete contribution
\[
 \frac{n!m}{24}
 \sum_{\substack{S\subseteq\Omega\setminus\{1\}\\|S|=k}}
   \left(\prod_{s\in S}s \xi_s^n\right)
   Q\left(\sum_{s\in S}\xi_s,\sum_{s\in S}\xi_s^2\right).
\]
The model now has $U=\Omega\setminus\{1\}$, $k=n-4$,
$\kappa=2$, $w_s=s \xi_s^n$, $\mathbf v_s=(\xi_s,\xi_s^2)$, and
the polynomial $Q$ above. The two components of $\mathbf v_s$
record the contributions to $\sum \xi_s$ and $\sum \xi_s^2$.
Writing $Q(X,Y)=\sum_{i,j}q_{i,j}X^iY^j$, the contribution is
\[
 \frac{n!m}{24}\sum_{i+2j\leq2n}q_{i,j}\,i!j!E_{n-4,(i,j)},
\]
where
\[
 E_{k,(i,j)}=[v^k z_1^i z_2^j]
       \prod_{s\in U}(1+v s \xi_s^n e^{z_1\xi_s+z_2\xi_s^2}).
\]
Only indices $(i,j)$ satisfying $i+2j\leq2n$ are needed:
assigning weights $1$ and $2$ to $X$ and $Y$ gives $Q$
weighted degree $2n$. This index set is downward-closed,
and $p>2n$ makes all factorials in~\eqref{eq:agg-exp-update}
invertible.

The implementation uses these constructions for the eight patterns
\begin{equation}\label{eq:agg-patterns}
 \mathcal G=\{(2),(3),(4),(2,2),(3,2),(2,2,2),
                    (3,2,2),(2,2,2,2)\}.
\end{equation}
Only multiplicities greater than $1$ are shown; the remaining
entries are distinct singleton roots. Patterns whose parts sum
to more than $n$ are omitted. For a nonempty unordered repeated-multiplicity
pattern $\boldsymbol\lambda=(\lambda_1,\ldots,\lambda_t)$, with $\lambda_i\geq2$,
put $\delta(\boldsymbol\lambda)=\sum_i(\lambda_i-1)$.
The selection includes all such patterns with
$\delta(\boldsymbol\lambda)\leq3$ and two with
$\delta(\boldsymbol\lambda)=4$. A pattern has
$n-\delta(\boldsymbol\lambda)$ distinct roots, so fixing a marked
root leaves an assignment count of degree
$n-\delta(\boldsymbol\lambda)-1$ in $m$. The selection therefore
removes the patterns with the largest enumeration degrees,
while using local Todd coefficients of degrees 0, 1, or 2.

Each class formula follows from the local Todd identity,
the multinomial expansion, and the finite-product calculation.
The marking lemma supplies its scaling weight. Every pattern
outside $\mathcal G$ is evaluated by the general Todd
recurrence~\eqref{eq:todd-exp-recurrence} and summed according to
Corollary~\ref{cor:max-anchor}; multisets with all roots
distinct contribute $0$. Thus the class sums and the remaining
multisets together give the complete root average. Further
patterns can be treated by the same exponential-moment lemma, but
their usefulness depends on the number and degrees of the
required moment variables and the cost of the resulting arrays.

\section{Exact algorithms, complexity, and implementation}\label{sec:algorithms}

We assemble the formulas of Section~\ref{sec:method} into an exact
computation. A combinatorial bound supplies the stopping criterion
for Chinese remainder reconstruction. We then analyze the cost of
the modular sums, describe their implementation, and illustrate
the passage from local Todd coefficients to a complete series.

\subsection{Exact modular algorithms}\label{sec:crt}

Integer recovery requires an upper bound on the desired count.
Choosing all but the last row gives the following bound.

\begin{lemma}\label{lem:bound}
For every $n\geq2$ and $r\geq0$,
\begin{equation}\label{eq:bound}
 0\leq H_n(r)\leq U_n(r):=
 \binom{r+n-1}{n-1}^{n-1}.
\end{equation}
\end{lemma}
\begin{proof}
A row of sum $r$ has $\binom{r+n-1}{n-1}$ nonnegative integer realizations.
Choose the first $n-1$ rows arbitrarily. Each column sum then uniquely
determines the last entry in that column; this either gives a valid last
row or produces a negative entry. Thus each choice of the first $n-1$
rows gives at most one matrix counted by $H_n(r)$.
\end{proof}

Choose distinct primes satisfying~\eqref{eq:prime}, compute the corresponding
residues, and continue until their product $P$ satisfies
\begin{equation}\label{eq:crtbound}
 P>U_n(r).
\end{equation}
The Chinese remainder theorem gives a unique representative $C\in[0,P)$.
Because $H_n(r)\in[0,U_n(r)]$, it follows that $C=H_n(r)$. A further prime,
excluded from $P$, provides an implementation check. It is not needed for
mathematical uniqueness and is not a substitute for~\eqref{eq:crtbound}.
The existence of arbitrarily many primes $1\pmod m$ is standard; see,
for example, the treatment of Dirichlet's theorem in~\cite{Apostol}.
For $m=1$ any sufficiently large prime is admissible.

We now describe how intermediate calculations are reused during
the enumeration. For each remaining multiplicity pattern, fix a maximal part at the
root $1$ and arrange the remaining parts in nonincreasing order.
Assign distinct roots to these parts; within each block of equal
parts their root indices increase. At each step, retain only indices
leaving enough unused larger roots to complete the current
equal-multiplicity block. Thus the repeated-root groups
precede the singleton groups.  Every multiset in $\mathcal M_{\max}$ has a unique
such assignment, and every partial assignment is a prefix of it.

For a repeated root $a$ already assigned to a part $\mu_a$, let
$N_a$ contain all currently assigned entries
other than the $\mu_a$ copies of $a$, with their multiplicities.
Store
\begin{align*}
 P_a&=\prod_{b\in N_a}(1-b/a)^{-1},\\
 c_{a,j}&=j\left(-r\mathbf1_{j=1}+(\mu_a-m^j)\beta_j
                   +\mathcal S_j(N_a/a)\right),
                       \qquad 1\leq j\leq \mu_a-2.
\end{align*}
When a group $(b,\mu_b)$ is appended, every existing repeated root
updates by
\[
 P_a\leftarrow P_a(1-b/a)^{-\mu_b},\qquad
 c_{a,j}\leftarrow c_{a,j}+\mu_b j C_j(y_{b/a}).
\]
If $\mu_b\geq2$, initialize the new root's arrays from the earlier
groups by the same displayed definitions.  All multiplicities
and coefficient ranges have been fixed before this traversal.
Induction on the number of assigned groups proves the stored
identities at every prefix.  At a leaf,
$C_a=(-1)^{\mu_a}mP_a/a$ is the scalar in
\eqref{eq:todd-prefactor}, $c_{a,j}=j\ell_{a,j}$, and the exponential
recurrence becomes
\[
 f_{a,j}=\frac1j\sum_{i=1}^{j}c_{a,i}f_{a,j-i}.
\]
Thus common prefixes share the construction of all local logarithms
and prefactors, while the terminal calculation is exactly the Todd
formula.  Restoring the saved parent state on return makes the
updates independent of the traversal order.

The sum of $\mu_a-1$ over repeated roots is at most $n$.
Along a complete root-to-leaf path, updating the existing arrays
and initializing all new ones therefore uses $O(n^2)$ field
operations in total; the leaf exponentiation has the same bound.
Summing the path bounds over the leaves is an upper bound for
the whole traversal because every enumerated prefix is extendible.
Consequently prefix reuse preserves the $O(n^2R(n,m))$ worst-case
bound while eliminating repeated preparation of common prefixes.
Here $R(n,m)$ is the number of multisets in $\mathcal M_{\max}$ with multiplicity patterns outside $\mathcal G$,
counted explicitly in Subsection~\ref{sec:computation}.

The implementation selects $p>2n$ and $m\mid p-1$.
These conditions make every factorial required by
Lemma~\ref{lem:exponential-moments} invertible for the selected
classes. The finite-field tables are prepared once per modulus
and shared among workers;
each worker retains its own traversal state and partial sum.

\medskip
\Needspace{10\baselineskip}
\noindent\emph{Algorithm 1: Exact evaluation of $H_n(r)$ by Todd coefficients.}
\begin{enumerate}[label=(S\arabic*),ref=S\arabic*]
 \item\label{step:init} If $n=1$, return $1$; if $n=2$, return $r+1$; if $r=0$,
 return $1$. Otherwise set $m=r+1$, compute $U_n(r)$, and initialize
 the Chinese remainder modulus $P=1$ and representative $C=0$.
 \item\label{step:prime} Select an unused prime $p>2n$ with $m\mid p-1$
 and construct $\Omega=\{1,\omega,\ldots,\omega^{m-1}\}$.
 Index the roots by powers of $\omega$ and precompute factorials,
 inverse differences, and the logarithmic ratio data $C_j(y_c)$
 from Subsection~\ref{sec:dd}. Prepare the multiplicity
 factors used in the prefix updates, and initialize $S=G=0$.
 \Needspace{10\baselineskip}
 \item\label{step:sum} Add the complete normalized contributions of
 the feasible patterns in the implemented set $\mathcal G$ of
 \eqref{eq:agg-patterns} to $G$, using
 Subsections~\ref{sec:agg-classes} and~\ref{sec:moments}.
 For every pattern outside
 $\mathcal G$, assign a maximal multiplicity to the marked root $1$ and traverse the distinct-root
 assignments with ordered equal-multiplicity blocks. For each leaf:
 \begin{enumerate}[label=(\alph*)]
  \item Use the cached weighted logarithmic coefficients
  $c_{a,j}=j\ell_{a,j}$ and prefactors to compute the Todd
  coefficients by~\eqref{eq:todd-exp-recurrence}, and sum the
  contributions in~\eqref{eq:todd-local} to obtain $h_r(M)$.
  \item Compute $\mult(M)$, $\eta(M)$, and
  $\Theta(M)=(\prod_a a^{\mu_a})h_r(M)^n$.
  \item Add $\mult(M)\Theta(M)/\eta(M)$ to $S$ in $\F_p$.
 \end{enumerate}
 \item\label{step:crt} Set $\rho_p=m^{1-n}S+G$. Combine $C\pmod P$ and
 $\rho_p\pmod p$ by the Chinese remainder theorem, replacing $P$ by $Pp$
 and $C$ by the least nonnegative representative modulo $Pp$.
 \item\label{step:stop} Repeat steps (\ref{step:prime})--(\ref{step:crt})
 until $P>U_n(r)$. If a checking prime is requested, repeat
 steps (\ref{step:prime}) and (\ref{step:sum}) at one further prime
 and compare $m^{1-n}S+G$ with the reduction of $C$; do not use
 this prime in reconstruction. Return $C$.
\end{enumerate}

\noindent\emph{Algorithm 2: Complete Ehrhart series.}
For $n\geq3$, use Algorithm 1 for $0\leq r\leq K$.
Apply~\eqref{eq:transform}, reflect the coefficients
by~\eqref{eq:pal}, and return~\eqref{eq:series}.
Formula~\eqref{eq:binomial} gives the counting polynomial if required.
An additional evaluation at $K+1$ tests a value unused in reconstruction.
The cases $n=1,2$ are given by their elementary formulas.

\begin{theorem}\label{thm:correctness}
Algorithm 1 returns $H_n(r)$ exactly, and Algorithm 2 returns the
complete Ehrhart series of $\B_n$.
\end{theorem}
\begin{proof}
The elementary branches are immediate. In every other case,
Theorem~\ref{thm:filter} expresses the required count modulo $p$
as a finite root average. The local formula in
Subsection~\ref{sec:dd} computes every $h_r(M)$ exactly and makes
the all-distinct terms vanish. Corollary~\ref{cor:max-anchor}
accounts for permutations and scaling, including nontrivial
stabilizers. The disjoint class sums in
Subsections~\ref{sec:agg-classes} and~\ref{sec:moments} give
their complete normalized
contributions. The ordered equal-multiplicity blocks enumerate
every remaining multiset in $\mathcal M_{\max}$ exactly once, and the prefix
invariant proves that its cached local value is the Todd value.
Their union is the complete root sum, so
step (\ref{step:crt}) obtains $H_n(r)\pmod p$.
The bound~\eqref{eq:bound} and the stopping condition
$P>U_n(r)$ make the retained integer unique. Arbitrarily many
admissible primes exist, so this stopping condition is eventually
met. Proposition~\ref{prop:reconstruct} then proves Algorithm 2.
\end{proof}

\subsection{Complexity and computational organization}\label{sec:computation}

Additions, multiplications, and divisions in $\F_p$ each count as
one field operation. Root indices are manipulated modulo $m$;
this integer bookkeeping is separate from the field-operation
count. We analyze the modular calculation once a primitive
$m$th root has been obtained. The logarithmic coefficients
in~\eqref{eq:todd-log} use the formulas
of~\cite[Equations~(6) and (7)]{XZZ25}. Their Stirling coefficients
require $O(n^2)$ operations; Horner evaluation for all root ratios
requires $O(mn^2)$. Caching the multiplicity factors and weighted
logarithmic coefficients fits within the same operation bound and
occupies $O(mn^2)$ field elements. The implementation also
prepares $O(m^2)$ inverse root differences.

For a repeated root $a$, the exponential recurrence
in~\cite[Lemma~4]{XZZ25} costs
$O((\mu_a-1)^2)$ operations; its characteristic condition holds
because $\mu_a-1<n<p$. Since
\[
 \sum_{\mu_a\geq2}(\mu_a-1)\leq n,\qquad
 \sum_{\mu_a\geq2}(\mu_a-1)^2\leq n^2,
\]
the complete local evaluation at a leaf costs $O(n^2)$.
The prefix argument in Subsection~\ref{sec:crt} bounds the
preparation over all extendible paths by the same order per leaf.
It remains to count these leaves and the work of the class sums.

Let $\boldsymbol{\lambda}=(\lambda_1,\ldots,\lambda_t)$ be a nonempty unordered repeated-multiplicity
pattern, with $\lambda_i\geq2$, and put
$k=n-\sum_i\lambda_i$.  Let $\nu_j$ be the number of entries equal to
$j$, and $q=\max_i\lambda_i$.  The exact number of multisets with an
eligible maximal-multiplicity mark fixed at $1$ is
\begin{equation}\label{eq:agg-pattern-count}
 N_{\boldsymbol{\lambda}}(n,m)=\binom{m-1}{t-1}\binom{m-t}{k}
       \frac{(t-1)!}{(\nu_q-1)!\prod_{j\ne q}\nu_j!}.
\end{equation}
Binomial coefficients are $0$ for infeasible choices.  Define the
residual count
\begin{equation*}
 R(n,m)=\sum_{\substack{\boldsymbol{\lambda}\notin\mathcal G\\
                         \lambda_i\geq2,\ \sum_i\lambda_i\leq n}}
                  N_{\boldsymbol{\lambda}}(n,m).
\end{equation*}
This is the number of individual multiset evaluations, not the total
arithmetic work.  The cost of the class sums must also be included.

Applying~\eqref{eq:moment-cost} to the choices of weights and
statistics in Subsections~\ref{sec:agg-classes} and~\ref{sec:moments}
gives the following conservative bounds, which include
the exponential-moment calculations and the
preparation of the corresponding weights:
\begin{center}
\begin{tabular}{@{}ll@{}}
\toprule
Repeated pattern & Field operations for its complete contribution\\
\midrule
$(2^t)$, $1\leq t\leq4$ &
$O\!\left(\binom{m-1}{t-1}\binom{n+t-1}{t-1}(m+1)(n+1)\right)$\\[2pt]
$(3,2^\nu)$, $0\leq\nu\leq2$ &
$O\!\left(\binom{m-1}{\nu}\binom{n+\nu}{\nu}(m+1)(n+1)^3\right)$\\[2pt]
$(4)$ & $O\!\left((m+1)(n+1)^5\right)$\\
\bottomrule
\end{tabular}
\end{center}
In the first bound, the binomial coefficient counting exponent
assignments may be replaced by the number retained by
maximum-exponent marking.  Denote the sum
of the bounds for the feasible patterns in the implemented set
$\mathcal G$ by $G(n,m)$.
With the cached ratio--multiplicity tables and a dense table of inverse
root differences, a complete modular count requires
\begin{equation*}
 O\bigl(m^2+mn^2+n^2R(n,m)+G(n,m)\bigr)
 \quad\text{field operations}.
\end{equation*}
Moment arrays may be reused between root sets and exponent
assignments; their largest storage bound is $O(n^3)$ for the
quadruple-root class.  The implementation also stores $J$
work items, each specifying a feasible residual multiplicity pattern
and the first assigned root other than $1$ (or the all-equal case).  Every work item
has a completion and distinct items have disjoint leaves, so $J\leq
R(n,m)$. The class with four double roots also stores
$\binom{m-1}{3}$ root triples for scheduling. Including this list,
the total serial working storage is
$O(m^3+mn^2+n^3+nJ)$ field elements and machine words. Each concurrent
worker needs its own moment and traversal arrays; the prepared
tables and work-item lists are shared. The field work for the initial
pattern weights is $O(nR(n,m))$ and fits within the stated modular
bound.  Generating multiplicity partitions also uses integer
bookkeeping: a conservative bound, including work-item preparation,
is $O(n^2\mathfrak p(n)+nJ+m^3)$, where
$\mathfrak p(n)$ is the number of integer partitions of $n$.
For fixed $n$, $J=O_n(m)$, so this scheduling does not change the
fixed-order bounds below.

To describe the dependence on $m$ for fixed order, write
$\delta(\boldsymbol{\lambda})=\sum_i(\lambda_i-1)$.  Formula~\eqref{eq:agg-pattern-count}
has degree $n-\delta(\boldsymbol{\lambda})-1$ in $m$ for each feasible pattern.
All patterns with $\delta\leq3$ occur in $\mathcal G$.
Therefore $R(n,m)=O_n(m^{n-5})$ for $n\geq5$, while $R(n,m)=0$
for $n\leq4$.  More explicitly, the remaining defect-$4$ patterns
are $(5)$, $(4,2)$, and $(3,3)$, so for $n\geq6$,
\[
 R(n,m)=\left(\frac1{(n-5)!}+\frac2{(n-6)!}\right)m^{n-5}
                       +O_n(m^{n-6}).
\]
Also $R(5,m)=1$.  Since $G(n,m)=O_n(m^4)$, the modular bound is
\begin{equation*}
 O_n\bigl(m^{n-5}+m^4\bigr)\qquad(n\geq5).
\end{equation*}
This fixed-order statement does not give a polynomial bound in $n$.
Let $L_r$ be the number of reconstruction primes for $H_n(r)$.
The modular part of complete series
reconstruction has cost
\begin{equation*}
 O\left(\sum_{r=1}^{K}L_r\left(
        (r+1)^2+(r+1)n^2+n^2R(n,r+1)+G(n,r+1)\right)\right).
\end{equation*}
Prime selection, primitive-root construction, and integer CRT are
additional costs. For a final CRT modulus $P_r$, put
$\chi_r=\lceil\log_2P_r\rceil+1$. Incremental reconstruction with
classical integer arithmetic has the conservative bound
$O(L_r\chi_r^2)$ bit operations. A checking prime adds one modular
evaluation, and an extra sample extends the sum to $K+1$.
The numerator transform uses $O(K^2)$ exact integer arithmetic
operations; the cases $n=1,2$ use their elementary formulas.

The computations reported here cover $3\leq n\leq14$, with
$0\leq r\leq79$, for Birkhoff counts, and $3\leq n\leq13$, with
$0\leq r\leq61$, for World Cup counts.
The program evaluates all residual local coefficients by the Todd
logarithm--exponential recurrence. Fixed field tables are shared;
independent repeated-root assignments have worker-local moment
arrays, traversal states, and modular partial sums. Scalar class
results and residual partial sums are combined before CRT.
Local identities can also be checked independently by extracting
coefficients from $\prod_{x\in M}(1-xu)^{-1}$.

\Needspace{7\baselineskip}
\subsection{A complete calculation at order $3$}\label{sec:example}

We illustrate Theorem~\ref{thm:anchor} in the maximal-multiplicity
form of Corollary~\ref{cor:max-anchor}, including the local Todd
evaluation and integer recovery, and then reconstruct the series.
Take $n=3$, $r=2$, $m=3$, and $p=13$. The root set is
$\Omega=\{1,3,9\}$. By the cancellation and maximum-multiplicity
marking, the relevant multisets are
$(1,1,1)$, $(1,1,3)$, and $(1,1,9)$.
For each of them $t(M)=\eta(M)=1$, so the two marking formulas
have the same weights in this example.

For $M=(1,1,1)$ the repeated root $1$ has $\mu_1=3$, so only a
degree-$1$ Todd coefficient is needed. With the normalized kernel $T(s)$,
\[
 h_2(1,1,1)
 =-3[s]\,e^{-2s}\frac{T(s)^3}{T(3s)}.
\]
Since $[s]\log T(s)=-1/2$, the logarithm of the normalized
factor has coefficient $-2+(3-3)(-1/2)=-2$. Its exponential
therefore has the same linear coefficient, and $h_2(1,1,1)=6$.
For $M=(1,1,a)$ with $a\in\{3,9\}$ the double root contributes
only a degree-$0$ coefficient. Thus
\[
 h_2(1,1,a)=\frac{3}{1-a}.
\]
The simple root contributes $0$. The complete data in $\F_{13}$ are
\begin{center}
\begin{tabular}{@{}lrrrr@{}}
\toprule
$M$ & $h_2(M)$ & $\Theta(M)$ & $\mult(M)$ & $\eta(M)$\\
\midrule
$(1,1,1)$ & $6$ & $8$ & $1$ & $1$\\
$(1,1,3)$ & $5$ & $11$ & $3$ & $1$\\
$(1,1,9)$ & $11$ & $6$ & $3$ & $1$\\
\bottomrule
\end{tabular}
\end{center}
Consequently,
\[
 H_3(2)=3^{-2}(8+3\cdot11+3\cdot6)=8\pmod{13}.
\]
The analogous residues at primes $7$ and $19$ are $0$ and $2$.
Here $U_3(2)=\binom42^2=36$. Since $7\cdot13=91>36$, the
first two residues recover $H_3(2)=21$ uniquely; the unused prime
$19$ confirms it.

Finally, $d=4$, $s=2$, and the numerator is $1+a_1z+z^2$.
The values $H_3(0)=1$ and $H_3(1)=6$ give $a_1=6-5=1$, hence
\[
 \Ehr_{\B_3}(z)=\frac{1+z+z^2}{(1-z)^5},\qquad
 H_3(r)=\frac{(r+1)(r+2)(r^2+3r+4)}8.
\]
The independently computed value $H_3(2)=21$ agrees with this
polynomial and was not used to determine the numerator.

\subsection{Computational results and normalized volumes}\label{sec:parameters}

The implementation carries out the finite-field calculations,
integer reconstruction, and numerator recovery described above.
Complete-series measurements use the full reconstruction range for
each order and include the arithmetic checks specified with the
timings.

The measurements in Table~\ref{tab:native-times} were obtained in
a Linux virtual machine identifying its processor as an AMD Ryzen
Threadripper 7980X, with 64 available logical processors, 32
guest-reported physical cores, and 62.2\,GiB of memory. The optimized
C++ implementation used 64 worker threads, with Python~3 for integer
reconstruction and numerator recovery. Each entry is the sum of the
wall-clock times of the distinct native modular evaluations, plus any
recorded interrupted-worker wall time, rather than a sum of per-thread
CPU times. The completed native evaluation times exclude Python
orchestration, process startup, file output, CRT, and numerator recovery.
For World Cup counts, $K_0$ is the reconstruction cutoff defined
in Section~\ref{sec:worldcup}.
Each complete-series run includes all samples $0\leq r\leq K$
(respectively $K_0$), one further sample at $K+1$ (respectively $K_0+1$),
and an additional checking prime for every nontrivial sample.
The reported times include table preparation, class sums, and residual
evaluation for these samples and checking primes.
The runs through order $13$ use no saved numerical evaluations as inputs.
The order-$14$ Birkhoff run resumes from its own saved checkpoints;
each saved evaluation is counted only once in the cumulative time.
The complete numerators for Birkhoff orders $10$--$14$ and World Cup orders
$9$--$13$ are listed in Appendices~\ref{app:birkhoff-series}
and~\ref{app:worldcup-series}, respectively.

\begin{table}[htbp]
\centering
\caption{Cumulative native computation time for complete Ehrhart-series runs (seconds unless otherwise indicated).}
\label{tab:native-times}
\begingroup
\footnotesize
\setlength{\tabcolsep}{2pt}
\begin{tabular}{@{}l*{12}{c}@{}}
\toprule
$n$ & 3 & 4 & 5 & 6 & 7 & 8 & 9 & 10 & 11 & 12 & 13 & 14\\
\midrule
Birkhoff & 0.012 & 0.025 & 0.047 & 0.097 & 0.220 & 0.963 & 3.773 & 19.071 & 117.022 & 1367.769 & 52256.236 & about $1$ month\\
World Cup & 0.006 & 0.023 & 0.080 & 0.206 & 0.440 & 1.007 & 1.968 & 5.063 & 61.199 & 1852.162 & 91373.351 & \textemdash\\
\bottomrule
\end{tabular}
\endgroup
\end{table}

The elementary orders $1$ and $2$ are not included in these measurements;
at order $1$ the World Cup counting series
is $1$, as explained in Section~\ref{sec:worldcup}.

For the Birkhoff polytopes at
orders $10$, $11$, $12$, $13$, and $14$, the series have the form
\begin{equation*}
 \Ehr_{\B_n}(z)=
 \frac{\sum_{j=0}^{(n-1)(n-2)}a_{n,j}z^j}{(1-z)^{(n-1)^2+1}},
 \qquad n\in\{10,11,12,13,14\}.
\end{equation*}
Their numerator degrees are $72$, $90$, $110$, $132$, and $156$, and their
denominator powers are $82$, $101$, $122$, $145$, and $170$.
For example, $a_{n,1}=n!-((n-1)^2+1)$ gives the first
nonconstant coefficients
\[
 \begin{aligned}
 &(a_{10,1},a_{11,1},a_{12,1},a_{13,1},a_{14,1})\\
 &\quad=(3\,628\,718,\ 39\,916\,699,\ 479\,001\,478,\ 6\,227\,020\,655,\ 87\,178\,291\,030).
 \end{aligned}
\]

Normalized volume is taken relative to the lattice in the affine
span of $\B_n$, with a unimodular simplex having volume $1$.
Writing $d=(n-1)^2$, the Ehrhart leading-coefficient
formula~\cite{BR} gives
\begin{equation*}
 \Vol(\B_n)=d!\,[r^d]H_n(r)=h_n^*(1)
 =\sum_{j=0}^{(n-1)(n-2)}a_{n,j}\qquad(n\geq2).
\end{equation*}
The point $\B_1$ has normalized volume $1$.
Thus the reconstructed numerator also determines the normalized
volume; the order-$3$ example gives $\Vol(\B_3)=3$.
Table~\ref{tab:normalized-volumes} gives the values through
order $14$.
Throughout the tables in this paper, an integer split across
lines is read by concatenating its digit strings.

\Needspace{19\baselineskip}
\begingroup
\small
\setlength{\tabcolsep}{5pt}
\renewcommand{\arraystretch}{1.1}
\begin{longtable}{@{}>{\raggedleft\arraybackslash}p{0.035\textwidth}>{\raggedleft\arraybackslash}p{0.045\textwidth}>{\ttfamily\fontsize{9}{11}\selectfont\raggedright\arraybackslash}p{\dimexpr0.92\textwidth-4\tabcolsep\relax}@{}}
\caption{Normalized volumes of the Birkhoff polytopes.}\label{tab:normalized-volumes}\\
\toprule
$n$ & $d$ & \multicolumn{1}{l}{$\Vol(\B_n)$}\\
\midrule\endfirsthead
\multicolumn{3}{l}{\small Table \thetable\ continued.}\\
\toprule
$n$ & $d$ & \multicolumn{1}{l}{$\Vol(\B_n)$}\\
\midrule\endhead
\midrule\multicolumn{3}{r}{\small Continued on the next page.}\\\endfoot
\bottomrule\endlastfoot
1 & 0 & \seqsplit{1}\\
2 & 1 & \seqsplit{1}\\
3 & 4 & \seqsplit{3}\\
4 & 9 & \seqsplit{352}\\
5 & 16 & \seqsplit{4718075}\\
6 & 25 & \seqsplit{14666561365176}\\
7 & 36 & \seqsplit{17832560768358341943028}\\
8 & 49 & \seqsplit{12816077964079346687829905128694016}\\
9 & 64 & \seqsplit{7658969897501574748537755050756794492337074203099}\\
10 & 81 & \seqsplit{5091038988117504946842559205930853037841762820367901333706255223000}\\
11 & 100 & \seqsplit{4850614184837422809226598217097242693744218917606734710902805771561681107752418130058980}\\
12 & 121 & \seqsplit{8294118713368905151290245946382878334338366795638682996497206621798886666629848232728011120424312160760328308352}\\
13 & 144 & \seqsplit{31146922969881028398746848185202547439267305465634372901511140324036051990801071776163168371235808731555571796519809780355974501248986006332}\\
14 & 169 & \seqsplit{308549755845407894488109885317227215067290432907208790705186279952837472685385945043457970373556933842222526754493971993748871019254746653099188375633416225750211674555264}\\
\end{longtable}
\endgroup

\Needspace{8\baselineskip}
\section{Application to the World Cup problem}\label{sec:worldcup}

For the World Cup problem, the diagonal entries are $0$, so the
generating polynomial for row $i$ omits the variable $x_i$.
The root filter and the reconstruction argument still apply.
The row polynomials can be evaluated using the Todd coefficients
already computed for $h_r(\mathbf{X})$.

\subsection{The counting model and its Ehrhart structure}

In the round-robin model of Ekhad and Zeilberger~\cite{EZ}, let $a_{ij}$
be the number of goals scored by team $i$ against team $j$. The teams are
labelled, each pair plays one match, and $a_{ii}=0$. Requiring every team
to score and concede exactly $r$ goals gives
\begin{equation*}
 D_n(r)=\#\bigl\{A=(a_{ij})\in\N^{n\times n}:
 a_{ii}=0,\quad \sum_j a_{ij}=r,\quad\sum_i a_{ij}=r\bigr\}.
\end{equation*}
Thus $(a_{ij},a_{ji})$ specifies the score of the match between $i$ and $j$.
The notation $D_n$ here denotes the function called $S_n$ in~\cite{EZ}.
For $r=1$, the matrices are permutation matrices
without fixed points, so $D_n(1)$ equals the number of derangements
of $n$ elements.

Write $\mathcal W_n=\{A\in\B_n:a_{ii}=0\text{ for all }i\}$.
For $n\geq2$, this is a nonempty lattice face of $\B_n$, and $D_n$
is its Ehrhart polynomial. In the elementary case $n=2$, the unique
matrix has off-diagonal entries $r$, so $D_2(r)=1$ and
$\Ehr_{\mathcal W_2}(z)=1/(1-z)$.
For $n=1$, the matrix count instead gives $D_1(0)=1$ and
$D_1(r)=0$ for $r>0$. We record its counting series
$F_1(z)=\sum_{r\geq0}D_1(r)z^r=1$ separately:
$\mathcal W_1$ is empty, so this is not the Ehrhart series of a
nonempty lattice polytope.
For $n\geq3$, the dimension and reflection identities are
\begin{align*}
 d_0&=n^2-3n+1, \\
 D_n(-1)&=\cdots=D_n(-(n-2))=0,\qquad
 D_n(-(n-1)-r)=-D_n(r).
\end{align*}
These structural identities are given in~\cite[Comment~1]{EZ}.
At $n=2$, $\mathcal W_2$ is a point.

For $n\geq3$, the corresponding numerator parameters can be read
directly from this structure. Relative-interior lattice points have all
off-diagonal entries positive. Subtracting $J-I$, where $J$ is the
matrix with all entries $1$, decreases every row and column sum by $n-1$ and gives
an arbitrary nonnegative matrix with diagonal entries $0$. The codegree is therefore
$n-1$. These matrices are magic labelings of the bipartite graph obtained
from $K_{n,n}$ by deleting a fixed perfect matching. Thus
\cite[Proposition~4.5]{Stanley76} gives nonnegativity; equivalently, one may
apply \cite[Theorem~2.1]{Stanley80}. Together with Ehrhart reciprocity
\cite{BR}, this gives
\begin{equation*}
 \Ehr_{\mathcal W_n}(z)=\sum_{r\geq0}D_n(r)z^r
 =\frac{g_n(z)}{(1-z)^{d_0+1}},\qquad
 g_n(z)=\sum_{j=0}^{s_0}b_{n,j}z^j,
 \quad s_0=(n-1)(n-3),
\end{equation*}
where
\begin{equation}\label{eq:wc-pal}
 b_{n,j}\in\N,\qquad b_{n,0}=b_{n,s_0}=1,\qquad
 b_{n,j}=b_{n,s_0-j}.
\end{equation}
Indeed, the interior translation identifies the interior generating
function with $z^{n-1}\Ehr_{\mathcal W_n}(z)$, which yields the displayed
palindromicity by reciprocity. Hence only
$D_n(0),\ldots,D_n(K_0)$ are needed, where $K_0=\lfloor s_0/2\rfloor$.
The degree $s_0$ can be odd.

\subsection{Evaluation of the row polynomials}

For row $i$, write $h_r(\mathbf{X}\setminus\{x_i\})$ for the complete
homogeneous symmetric polynomial in the remaining $n-1$ variables.
When evaluated at a multiset, this notation removes one occurrence
of $x_i$, even if its value is repeated. Set $h_{-1}(\mathbf{X})=0$.
The generating function identity
\[
 \sum_{r\geq0}h_r(\mathbf{X}\setminus\{x_i\})u^r
 =(1-x_i u)\prod_j(1-x_j u)^{-1}
\]
gives
\begin{equation*}
 h_r(\mathbf{X}\setminus\{x_i\})=h_r(\mathbf{X})-x_i h_{r-1}(\mathbf{X}).
\end{equation*}
Thus, writing
\begin{equation*}
 F_{n,r}(\mathbf{X})=\prod_{i=1}^n\bigl(h_r(\mathbf{X})-x_i h_{r-1}(\mathbf{X})\bigr),
 \qquad
 D_n(r)=[x_1^r\cdots x_n^r]F_{n,r}(\mathbf{X}),
\end{equation*}
we again have a symmetric polynomial of total degree $nr$.
With $m=r+1$, the argument in Theorem~\ref{thm:filter} proves
\begin{equation}\label{eq:wc-filter}
 D_n(r)=m^{-n}\sum_{\mathbf{X}\in\Omega^n}
          \left(\prod_i x_i\right)F_{n,r}(\mathbf{X})
 \quad\text{in }\F_p.
\end{equation}

Thus the root sum requires the adjacent values $h_r(M)$ and
$h_{r-1}(M)$. The latter follows from the same local calculation:
for $r\geq1$, multiplying $E_M(u)$ by $u$ gives
\begin{equation}\label{eq:wc-second-ct}
 h_{r-1}(M)=\CT\limits_u
 \frac{u^{-(r-1)}-u^2}{\prod_b(1-bu)^{\mu_b}}.
\end{equation}
For $n\geq3$ this rational function is proper at infinity.
Its numerator vanishes at every $u=a^{-1}$, just as for $h_r$.
At a repeated root $a$, use the scalar $C_a$,
the normalized series $F_a(s)$, and its coefficients $f_{a,j}$
from Subsection~\ref{sec:dd}. Then
\begin{align}
 h_r(M)&=\sum_{a:\,\mu_a\geq2}C_a f_{a,\mu_a-2},
                    \label{eq:wc-todd-a}\\
 h_{r-1}(M)&=\sum_{a:\,\mu_a\geq2}\frac{C_a}{a}
                    \sum_{j=0}^{\mu_a-2}\frac{f_{a,j}}{(\mu_a-2-j)!}.
                    \label{eq:wc-todd-b}
\end{align}
\begin{proof}
Apply the same local change of variables to~\eqref{eq:wc-second-ct}.
Its numerator at $u=a^{-1}e^s$ equals
$a^{-2}e^{-(r-1)s}(1-e^{ms})$. Relative to the calculation for
$h_r(M)$, this inserts $a^{-1}e^s$. The local contribution is
therefore $(C_a/a)[s^{\mu_a-2}]e^sF_a(s)$, which is the convolution
in~\eqref{eq:wc-todd-b}. The same truncated-series reduction
is valid under $p>n$.
\end{proof}
Only a linear number of further operations per repeated root is
needed once the coefficients $f_{a,j}$ have been computed.
In particular, there is no second Todd exponentiation.
For $r=0$, use $h_0(M)=1$ and $h_{-1}(M)=0$ directly.

Define
\[
 \Theta_0(M)=\left(\prod_a a^{\mu_a}\right)
          \prod_a\bigl(h_r(M)-ah_{r-1}(M)\bigr)^{\mu_a}.
\]
\begin{proposition}\label{prop:wc-filter}
For $n\geq3$ and $r\geq1$,
\begin{equation}\label{eq:wc-anchor}
 D_n(r)=m^{1-n}\sum_{M\in\calM_{\max}}
                  \frac{\mult(M)}{\eta(M)}\Theta_0(M).
\end{equation}
The local evaluation of $\Theta_0(M)$ has the same asymptotic cost as
that of $\Theta(M)$ in Algorithm 1.
\end{proposition}
\begin{proof}
All-distinct multisets have $h_r(M)=h_{r-1}(M)=0$, by their local
expressions. The product $F_{n,r}$ is symmetric and homogeneous
of degree $nr$, so its weighted summand is invariant under
scaling by $\Omega$. Apply Lemma~\ref{lem:eligible-marking}
to roots of maximum multiplicity in~\eqref{eq:wc-filter}.
Equations~\eqref{eq:wc-todd-a} and~\eqref{eq:wc-todd-b} share all
Todd coefficient arrays. Their additional convolutions and the
row-product evaluation cost $O(n)$ operations per multiset,
within the stated bound.
\end{proof}

The shared Todd coefficients also allow selected World Cup classes
to be summed before the singleton roots are enumerated. The
implementation uses the repeated-multiplicity patterns
\begin{equation*}
 \mathcal G_0=\{(2),(3),(2,2),(3,2),(2,2,2),(2,2,2,2)\}.
\end{equation*}
As before, multiplicities not shown are $1$, and impossible patterns
are omitted. Patterns outside $\mathcal G_0$ are evaluated by
Proposition~\ref{prop:wc-filter}. The pattern $(2)$ contributes $0$.
For the other selected patterns, the following finite-product
construction gives the subset and moment recurrences.

Fix one of these multiplicity classes. Write a multiset in the class as
$M=R^{\boldsymbol{\mu}}\cup S$, where $R$ is the set of
repeated roots, $\boldsymbol{\mu}=(\mu_a)_{a\in R}$, and $R^{\boldsymbol{\mu}}$ contains
$\mu_a$ copies of each $a\in R$. The singleton set $S$ has size
$k=n-\sum_{a\in R}\mu_a$; put $t=|R|$. Set
\[
 V_a(S)=\prod_{s\in S}(a-s)^{-1}.
\]
For an indexed occurrence $x$ of $M$, the notation
$M\setminus\{x\}$ deletes that one occurrence. Suppose first that
all repeated roots are double, and put
\begin{equation}\label{eq:wc-agg-double-row}
 \widetilde q_a=a^{n-4}\prod_{b\in R\setminus\{a\}}(a-b)^{-2},
 \qquad P_{a,x}=a-x.
\end{equation}
The coefficient $f_{a,0}=1$ in the two adjacent Todd evaluations
gives
\begin{equation}\label{eq:wc-agg-row-linear}
 m^{-1}h_r(M\setminus\{x\})=\sum_{a\in R}\widetilde q_aV_a(S)P_{a,x}.
\end{equation}
For $|R|=1$, this vanishes when $x$ is the double root, proving
the assertion about the pattern $(2)$.

For the patterns $(3)$ and $(3,2)$, scale the unique triple root
to $1$ and let $\calD=R\setminus\{1\}$ be the set of double roots.
Thus $|\calD|$ is $0$ or $1$.  Define
\begin{align*}
 \widetilde q_1&=\prod_{b\in \calD}(1-b)^{-2},
 &\widetilde q_b&=b^{n-4}(b-1)^{-3}
             \prod_{a\in \calD\setminus\{b\}}(b-a)^{-2},\\
 c_0&=\frac{m+2n-7}{2}-2\sum_{b\in \calD}(1-b)^{-1},
 &\Lambda(S)&=\sum_{s\in S}(1-s)^{-1}.
\end{align*}
Thus $c_0=c-1$, where $c$ is the corresponding Birkhoff
triple-root parameter in Subsection~\ref{sec:moments}.
With an auxiliary variable $Z$, put
\begin{equation*}
 P_{1,x}(Z)=1+(1-x)Z,\qquad
 P_{b,x}(Z)=b-x\quad(b\in \calD).
\end{equation*}
Then the row formula is
\begin{equation}\label{eq:wc-agg-row-triple}
 m^{-1}h_r(M\setminus\{x\})
 =\left.\sum_{a\in R}\widetilde q_aV_a(S)P_{a,x}(Z)
                                      \right|_{Z=c_0-\Lambda(S)}.
\end{equation}
To derive it directly from the Todd coefficients, write
$\alpha=(m+2n-5)/2$.  The triple-root contributions to
$h_r(M)$ and $h_{r-1}(M)$ are respectively
\[
 m \widetilde q_1 V_1(S)(\alpha-2\sum_{b\in \calD}(1-b)^{-1}-\Lambda(S)),
 \qquad
 m \widetilde q_1 V_1(S)(\alpha-1-2\sum_{b\in \calD}(1-b)^{-1}-\Lambda(S)).
\]
Subtracting $x$ times the second contribution from the first gives
the first summand in
\eqref{eq:wc-agg-row-triple}.  The double-root contributions give
the remaining summands by~\eqref{eq:wc-agg-double-row}.

Here is the common subset construction.  Let $\mathbf{Y}=(Y_a)_{a\in R}$,
and define the polynomial for the rows corresponding to repeated roots
\begin{equation*}
 \mathcal P_R(\mathbf{Y},Z)=
     \prod_{b\in R}\left(\sum_{a\in R}P_{a,b}(Z)Y_a\right)^{\mu_b}.
\end{equation*}
When all roots are double, omit $Z$ throughout.  For an exponent
assignment $\mathbf{e}\in\N^R$ with $|\mathbf{e}|=n$, put
\[
 w_s(\mathbf{e})=s\prod_{a\in R}(a-s)^{-e_a}.
\]
In the all-double case define
\begin{equation}\label{eq:wc-agg-double-contraction}
 \Psi_R(\mathbf{e})=[v^k\mathbf{Y}^{\mathbf{e}}]\,
 \mathcal P_R(\mathbf{Y})
 \prod_{s\in\Omega\setminus R}
       \left(1+v w_s(\mathbf{e})\sum_{a\in R}(a-s)Y_a\right).
\end{equation}
For a class with a triple root, a moment variable records $\Lambda(S)$.
For polynomials in $Z$ and a truncated variable $z$, define the
linear functional
\begin{equation*}
 \mathcal L_{c_0} F(Z,z)=
   \sum_{a\geq0}\sum_{j=0}^{a}
     (-1)^j\frac{a!}{(a-j)!}c_0^{a-j}[Z^a z^j]F(Z,z).
\end{equation*}
It satisfies
$\mathcal L_{c_0}(P(Z)e^{z\lambda})=P(c_0-\lambda)$ by the binomial theorem.
With $\xi_s=(1-s)^{-1}$, put
\begin{equation}\label{eq:wc-agg-triple-contraction}
 \Psi_R(\mathbf{e})=\mathcal L_{c_0}\,[v^k\mathbf{Y}^{\mathbf{e}}]\,
 \mathcal P_R(\mathbf{Y},Z)
 \prod_{s\in\Omega\setminus R}
 \left(1+v w_s(\mathbf{e})e^{z \xi_s}
                    \sum_{a\in R}P_{a,s}(Z)Y_a\right).
\end{equation}
All exponentials here are finite truncations.  The largest required
$Z$ and $z$ degrees are $n-3$: for $(3)$ the degree is at most
$k=n-3$, and for $(3,2)$ it is at most $k+2=n-3$.
Consequently $p>n$ makes every required factorial invertible.
These products are instances of
Lemma~\ref{lem:exponential-moments}; the choice $p>2n$ in
Algorithm 1 also satisfies its factorial conditions here.

For a fixed set of repeated roots, expand the product of the row
polynomials using~\eqref{eq:wc-agg-row-linear} or
\eqref{eq:wc-agg-row-triple}.  The exponent $e_a$ counts how often
the summand at the root $a$ is chosen.  Pulling out
$\prod_a \widetilde q_a^{e_a}V_a(S)^{e_a}$ places the singleton denominator
factors, together with their root-product factor $s$, in $w_s(\mathbf{e})$.
Selecting the $v$-term exactly $k$ times chooses the singleton set;
the coefficient $\mathbf{Y}^{\mathbf{e}}$ counts the choices of summands in all rows.
For a triple root, the functional substitutes $c_0-\Lambda(S)$.
Thus $\Psi_R(\mathbf{e})$ is exactly the required sum over singleton sets.
There is no additional multinomial coefficient: all choices have
already been counted by the polynomial products.

In the following class sum, all $\mu_a$ equal $2$ for an all-double
class; for a triple-root class, $\mu_1=3$ and $\mu_a=2$ for
$a\in R\setminus\{1\}$. The value $\eta(M)$ is constant within the
fixed class; here $M$ can be any representative of that class.
The complete normalized class contribution is therefore
\begin{equation}\label{eq:wc-agg-class-total}
 \frac{n!m}{\eta(M)}
 \sum_{\substack{R\subseteq\Omega\\ \lvert R\rvert=t,\;1\in R}}
    \frac{\prod_{a\in R}a^{\mu_a}}{\prod_{a\in R}\mu_a!}
    \sum_{|\mathbf{e}|=n}\left(\prod_{a\in R}\widetilde q_a^{e_a}\right)\Psi_R(\mathbf{e}).
\end{equation}
The product of multiplicity factorials is constant within a class.
For the two triple-root classes, $1$ is the unique triple root and
$\eta(M)=1$.  For a class of $t$ double roots, require $1\in R$
and take $\eta(M)=t$.  These factors follow from the eligible-root
marking argument.  The factor $m$ combines $m^n$ from the row factors
with the marked root-average factor $m^{1-n}$.

For three or four double roots the expansion can instead mark a
root where $e_a$ is maximal.  Replace the outer factor $1/t$ by
restricting the inner sum to $e_1\geq e_a$ for all $a\in R$ and
dividing each term by
$\tau(\mathbf{e})=\#\{a:e_a=\max_b e_b\}$.  Each expanded term is invariant
under root scaling.  Indeed, if $k=n-2t$, the degrees contributed
by the repeated-root product, $\mathcal P_R$, the factors $\widetilde q_a^{e_a}$,
and the selected singleton factors are respectively
$2t$, $2t$, $n(k-2)$, and $k(2-n)$, whose sum is $0$.
This proves the alternative marking formula even when the scaling
action has stabilizers.

The finite products are evaluated each singleton candidate in turn,
retaining only the required coefficient indices.  In the all-double
case each factor has the form $1+$ a linear polynomial in $\mathbf{Y}$.
For the triple-root classes the same update is combined with the
univariate exponential-moment convolution.  This uses precisely
the finite subset and moment operations from the ordinary Birkhoff
calculation; the row polynomials differ because a variable is deleted.

\subsection{Reconstruction of the World Cup series}

We complete the calculation by recovering the integer counts
from~\eqref{eq:wc-anchor}. The first $n-1$ rows each have $\binom{r+n-2}{n-2}$
nonnegative assignments to their allowed entries. The column
sums determine the last row, giving
\begin{equation*}
 0\leq D_n(r)\leq U_n^0(r):=\binom{r+n-2}{n-2}^{n-1}.
\end{equation*}
For $n\geq3$, Algorithm 1 uses this bound, the class set
$\mathcal G_0$ with the contributions in~\eqref{eq:wc-agg-class-total},
and $\Theta_0(M)$ in place of $\Theta(M)$ on the residual multisets.
For $n\leq2$, use the elementary formulas above. The shared
Todd evaluation and the class-summation proof give the exact
modular count; the identical CRT argument proves integer recovery. The numerator is reconstructed from
$0\leq r\leq K_0=\lfloor s_0/2\rfloor$ by
\begin{equation*}
 b_{n,k}=\sum_{i=0}^k(-1)^{k-i}\binom{d_0+1}{k-i}D_n(i),
 \qquad0\leq k\leq K_0,
\end{equation*}
followed by the reflection~\eqref{eq:wc-pal}.
Consequently,
\begin{equation*}
 D_n(r)=\sum_{j=0}^{s_0}b_{n,j}\binom{r+d_0-j}{d_0}.
\end{equation*}
This proves correctness of the World Cup version of both
algorithms. Its reconstruction range is shorter, while its outer
sum has the separate complexity described below.

For instance, at order $4$, $d_0=5$ and $s_0=3$.
The values $D_4(0)=1$ and $D_4(1)=9$ give
\[
 \Ehr_{\mathcal W_4}(z)=\frac{1+3z+3z^2+z^3}{(1-z)^6},\qquad
 D_4(r)=\frac{(r+1)(r+2)(2r+3)(r^2+3r+5)}{30},
\]
in agreement with~\cite{EZ}.
Appendix~\ref{app:worldcup-series} gives the complete numerator
coefficients for orders $9$--$13$.

We now estimate the computational cost. Define $N_{\boldsymbol{\lambda}}(n,m)$
by~\eqref{eq:agg-pattern-count}, and let
\begin{equation*}
 R_0(n,m)=\sum_{\substack{\boldsymbol{\lambda}\notin\mathcal G_0\\
                            \lambda_i\geq2,\ \sum_i\lambda_i\leq n}}
                               N_{\boldsymbol{\lambda}}(n,m).
\end{equation*}
The residual traversal uses the same cached Todd logarithms and
prefactors as the Birkhoff calculation.  Once the local coefficients
have been computed, the convolution giving $h_{r-1}$ and the product of the row
polynomials require $O(n)$ additional operations.  The residual total
therefore costs $O(n^2R_0(n,m))$ field operations.

Let $G_0(n,m)$ denote the cost of the five nonzero class sums.
For fixed $n$, a class with $t$ repeated roots has $O_n(m^{t-1})$
repeated-root sets in which $1$ has maximal multiplicity.  Its coefficient arrays and the number of
exponent assignments depend only on $n$, while each product in
\eqref{eq:wc-agg-double-contraction} or
\eqref{eq:wc-agg-triple-contraction} is evaluated in a single pass over
at most $m$ singleton candidates.  Hence that class costs $O_n(m^t)$,
and $G_0(n,m)=O_n(m^4)$.  The complete modular cost, including the
precomputed tables, is
\begin{equation}\label{eq:wc-agg-total-cost}
 O\bigl(m^2+mn^2+n^2R_0(n,m)+G_0(n,m)\bigr).
\end{equation}
All patterns of defect $\sum_i(\lambda_i-1)$ at most $2$ are removed.
The only remaining defect-$3$ pattern is $(4)$, and therefore,
for fixed $n\geq5$,
\begin{equation*}
 R_0(n,m)=\frac{m^{n-4}}{(n-4)!}+O_n(m^{n-5}).
\end{equation*}
Also $R_0(3,m)=0$ and $R_0(4,m)=1$.
It follows that the optimized World Cup calculation has the
fixed-order bound
\begin{equation*}
 O_n\bigl(m^{n-4}+m^4\bigr)\qquad(n\geq4).
\end{equation*}
This bound includes both the residual evaluations and class sums;
it is not a polynomial bound in the variable $n$.

The coefficient arrays for the mixed pattern $(3,2)$ have four
indices (subset cardinality, summand choice, polynomial degree,
and moment degree), giving the conservative bound $O(n^4)$ per
worker. If $J_0$ is the residual work-item count, then
$J_0\leq R_0(n,m)$. The class with four double roots also stores
$O(m^3)$ scheduling triples. For $t=3,4$, the program caches the
transition and final-state indices of each exponent assignment in
the $t$-double-root class. Their total number $E_0$ satisfies
\[
 E_0=O\left(\sum_{t=3}^{4}
       \binom{n+t-1}{t-1}\binom{n-t}{t}\right)=O(n^7):
\]
there are $O\!\left(\binom{n+t-1}{t-1}\right)$ exponent assignments,
each with at most $\binom{k+t}{t}=\binom{n-t}{t}$ states, where
$k=n-2t$. Infeasible classes are skipped before constructing these
caches. The World Cup implementation therefore uses
$O(m^3+mn^2+n^4+nJ_0+E_0)$ field elements and machine words
for the shared data and one worker. Constructing the transition indices takes
$O(E_0)$ integer operations. The same partition-generation
bookkeeping described in Subsection~\ref{sec:computation} applies.
For fixed $n$, these caches have constant size and do not alter
the stated dependence on $m$.
For complete-series complexity, sum~\eqref{eq:wc-agg-total-cost}
over the reconstruction primes and $1\leq r\leq K_0$, then add
the CRT and numerator-transform costs as in
Subsection~\ref{sec:computation}.

\Needspace{8\baselineskip}

\section{Concluding remarks}\label{sec:conclusion}
We have developed an exact method for computing the Ehrhart series
of Birkhoff polytopes, proved its correctness, and analyzed its
arithmetic complexity. The method also applies to the World Cup
problem. Our computations determine the complete Birkhoff series
through order $14$ and the World Cup series through order $13$,
extending the previously known results to orders $10$--$14$ for
Birkhoff polytopes and orders $9$--$13$
for the World Cup problem. The same series also yields the normalized
volumes of the Birkhoff polytopes.

The main counting formula expresses $H_n(r)$ as a weighted sum
over root multisets. It combines a finite-field root-of-unity filter
with permutation and scaling symmetries. The key local reduction is
a constant-term cancellation: when the roots have order $m=r+1$,
only repeated roots contribute, and a root of multiplicity
$\mu_a\geq2$ requires a generalized Todd coefficient of degree
$\mu_a-2$. Thus the degree needed for each local calculation is
controlled by root multiplicity rather than by $r$. Shared
logarithmic coefficients reduce repeated work, while elementary
symmetric functions and finite-product identities sum selected
multiplicity classes. The remaining multisets are evaluated
individually. The Chinese remainder theorem recovers the integer
counts, and the symmetry of the Ehrhart numerators determines the
complete series from finitely many such counts.

The multiplicity classes selected in the implementation are
computational choices within a general summation procedure. Further
class sums can be derived from the same Todd formulas and
finite-product identities. Their practical value depends on the cost
and storage of the required coefficients compared with the individual
evaluations they replace. Selecting additional classes on this basis
may make computations at larger orders more efficient. It would also
be interesting to extend the method to other families of polytopes
and to study the optimal selection of multiplicity classes.

\subsection*{Acknowledgements}
Guoce Xin was partially supported by the National Natural Science Foundation of China (No. 12571355). Chen Zhang was partially supported by the National Natural Science Foundation of China (No. 12601634), and the Postdoctoral Fellowship Program and China Postdoctoral Science Foundation (No. BX20250066). Yueming Zhong was partially supported by the National Natural Science Foundation of China (No. 12401441), and the Natural Science Foundation of Hunan Province (No. 2025JJ60010).

\clearpage
\appendix
\setlength{\LTcapwidth}{\textwidth}
\section{Ehrhart series of the Birkhoff polytopes}\label{app:birkhoff-series}

Tables~\ref{tab:series10}--\ref{tab:series14} specify the complete
Ehrhart series for $10\leq n\leq14$. Write
\[
 \Ehr_{\B_n}(z)=\frac{h_n^*(z)}{(1-z)^{(n-1)^2+1}},
 \qquad h_n^*(z)=\sum_{j=0}^{(n-1)(n-2)}a_{n,j}z^j.
\]
The caption of each table specifies its order, denominator, and numerator
degree. The coefficients satisfy $a_{n,j}=a_{n,(n-1)(n-2)-j}$. A row containing
two indices gives the common coefficient at both indices; the middle
coefficient has a single index.
All entries are exact decimal integers; the convention for integers
split across lines is the same as in Table~\ref{tab:normalized-volumes}.

\begingroup
\small
\setlength{\tabcolsep}{5pt}
\renewcommand{\arraystretch}{1.05}
\begin{longtable}{@{}>{\raggedleft\arraybackslash}p{0.075\textwidth}>{\ttfamily\fontsize{8}{10}\selectfont\raggedright\arraybackslash}p{\dimexpr0.925\textwidth-2\tabcolsep\relax}@{}}
\caption{Ehrhart series of $\B_{10}$: $\Ehr_{\B_{10}}(z)=h_{10}^*(z)/(1-z)^{82}$ with $\deg_z h_{10}^*(z)=72$.}\label{tab:series10}\\
\toprule $j$ & \multicolumn{1}{l}{$a_{10,j}$}\\\midrule\endfirsthead
\multicolumn{2}{l}{\small Table \thetable\ continued ($n=10$).}\\
\toprule $j$ & \multicolumn{1}{l}{$a_{10,j}$}\\\midrule\endhead
\midrule\multicolumn{2}{r}{\small Continued on the next page.}\\\endfoot
\bottomrule\endlastfoot
0, 72 & \seqsplit{1}\\
1, 71 & \seqsplit{3628718}\\
2, 70 & \seqsplit{3930432549921}\\
3, 69 & \seqsplit{523531572961727040}\\
4, 68 & \seqsplit{14313685887007110389910}\\
5, 67 & \seqsplit{117110325931468069170480516}\\
6, 66 & \seqsplit{370305244152143622256665797598}\\
7, 65 & \seqsplit{540741822697742746416903002112576}\\
8, 64 & \seqsplit{414516895665374074638648683187039105}\\
9, 63 & \seqsplit{183436191001675125693891761801489869870}\\
10, 62 & \seqsplit{50376067327592353102629199786538999629626}\\
11, 61 & \seqsplit{9082439495909559883277485441555522339980638}\\
12, 60 & \seqsplit{1124124685278142653960391967009294709584743701}\\
13, 59 & \seqsplit{99016782897821728221462546352511822310001317720}\\
14, 58 & \seqsplit{6392862947856073623390254478279202126825536168180}\\
15, 57 & \seqsplit{310021821817414906733897940054396588635948823346176}\\
16, 56 & \seqsplit{11526467679848805121871838320696660355346099549753848}\\
17, 55 & \seqsplit{334296371158435459524826922243499569956294852228972416}\\
18, 54 & \seqsplit{7675614222160796755003973368449618459074323445142490285}\\
19, 53 & \seqsplit{141301368687456132994439669280227936421552014583034368390}\\
20, 52 & \seqsplit{2108549383764593830627825067587000240244001328139457833226}\\
21, 51 & \seqsplit{25748180019959407071117875946432501465130727937488807329358}\\
22, 50 & \seqsplit{259430644949146564878308595241968982518634659112149874067931}\\
23, 49 & \seqsplit{2172388251743737662976967496690261563624933447979318986045100}\\
24, 48 & \seqsplit{15213521184235022347184247187353961179574591356974405846418000}\\
25, 47 & \seqsplit{89596040968395980892894519522306777205649880259593298571821068}\\
26, 46 & \seqsplit{445864318945235784070572698794699448592379804155163391306946624}\\
27, 45 & \seqsplit{1882742376112666314159647820535412450893557211822803699825096628}\\
28, 44 & \seqsplit{6770679139870199435001217157342183415397556288627522120943984795}\\
29, 43 & \seqsplit{20801186277673884492693106818083073538553183319976981197958357830}\\
30, 42 & \seqsplit{54742091855624509891127343422350118666254248373865808932228737218}\\
31, 41 & \seqsplit{123683581502889715478428796195535625877431628348165441943384644054}\\
32, 40 & \seqsplit{240364207903588546157150876458942846384432519001416173834835524198}\\
33, 39 & \seqsplit{402387255997473740358035188547450828384139006678958357432144712190}\\
34, 38 & \seqsplit{580941927300599047625066955956145113149078451784348342226312738185}\\
35, 37 & \seqsplit{723913186556189811061909774320178360392903263805470299534099720848}\\
36 & \seqsplit{778959002885366533897754535827961188980182747256476613358558536024}\\
\end{longtable}
\endgroup

\clearpage
\begingroup
\small
\setlength{\tabcolsep}{5pt}
\renewcommand{\arraystretch}{1.05}
\begin{longtable}{@{}>{\raggedleft\arraybackslash}p{0.075\textwidth}>{\ttfamily\fontsize{7.6}{10}\selectfont\raggedright\arraybackslash}p{\dimexpr0.925\textwidth-2\tabcolsep\relax}@{}}
\caption{Ehrhart series of $\B_{11}$: $\Ehr_{\B_{11}}(z)=h_{11}^*(z)/(1-z)^{101}$ with $\deg_z h_{11}^*(z)=90$.}\label{tab:series11}\\
\toprule $j$ & \multicolumn{1}{l}{$a_{11,j}$}\\\midrule\endfirsthead
\multicolumn{2}{l}{\small Table \thetable\ continued ($n=11$).}\\
\toprule $j$ & \multicolumn{1}{l}{$a_{11,j}$}\\\midrule\endhead
\midrule\multicolumn{2}{r}{\small Continued on the next page.}\\\endfoot
\bottomrule\endlastfoot
0, 90 & \seqsplit{1}\\
1, 89 & \seqsplit{39916699}\\
2, 88 & \seqsplit{452781290674450}\\
3, 87 & \seqsplit{477314825689047718350}\\
4, 86 & \seqsplit{82809782801129946796143525}\\
5, 85 & \seqsplit{3615073965406805154890689634967}\\
6, 84 & \seqsplit{53133941556770686186206623242160448}\\
7, 83 & \seqsplit{322765899360273298659780109732006148760}\\
8, 82 & \seqsplit{940616252129049970157847218118695654897325}\\
9, 81 & \seqsplit{1470238409342476768006056617061512716051912025}\\
10, 80 & \seqsplit{1342589058994902212160998774084062339003485598198}\\
11, 79 & \seqsplit{765834182573489012126835860555662733351102630081785}\\
12, 78 & \seqsplit{287815191933010511358599421878416033591125901022460032}\\
13, 77 & \seqsplit{74407790512327740725062855401107779490260018421992616835}\\
14, 76 & \seqsplit{13709675629853852932473692551293899369698317944201973884250}\\
15, 75 & \seqsplit{1854028187754163794323929868692042083064388543087705907592899}\\
16, 74 & \seqsplit{188632409841035172379585622038900969080138784675693186012375410}\\
17, 73 & \seqsplit{14744386049585446730753287712110482685674612390506448037752006741}\\
18, 72 & \seqsplit{901426976909789176793811134260465190936126686733668551798031796550}\\
19, 71 & \seqsplit{43775986242066536452635125838818737704449431759034695657811679096775}\\
20, 70 & \seqsplit{1711448040433702770920608560344943552651403117609461035389084533362965}\\
21, 69 & \seqsplit{54499856970495985905642547037743573717677667437781962028854826290464135}\\
22, 68 & \seqsplit{1428204868712987759120339972294323776471884539993229753266854055160804835}\\
23, 67 & \seqsplit{31080025427426744413251743594100648816855537786999642781828643913321034515}\\
24, 66 & \seqsplit{566173063737500192621065134237927047053947432494028567693389635355645489600}\\
25, 65 & \seqsplit{8695345806115821490029509591382550275764459941144423800023560761199008535092}\\
26, 64 & \seqsplit{113305285654904964932147285250132430252821437040823287609192118527531583396749}\\
27, 63 & \seqsplit{1259797151109479867098510975938134690737958990497517040617077700510099987516160}\\
28, 62 & \seqsplit{12012707324997888260860623532190304778818786069924094595994010622386759867241229}\\
29, 61 & \seqsplit{98683534126789874571172345279300121908627613585575305532191285476151859650849225}\\
30, 60 & \seqsplit{701264525743973464419571813799474421854490638323460752838501689411730995586717258}\\
31, 59 & \seqsplit{4326529841872890917122409266350815386088271117819338335846398259248633456581764174}\\
32, 58 & \seqsplit{23250964630293294799632948567746761087600315419533549529744067637604639901058577430}\\
33, 57 & \seqsplit{109158954981761733902611836574450456428990580971275045170837816274734812003534714525}\\
34, 56 & \seqsplit{448882446790010834786757429962709617063839598673957729896362240820267957485025622868}\\
35, 55 & \seqsplit{1620592310112405171957487595236473857264868643198282109588715489211303406889964855579}\\
36, 54 & \seqsplit{5147315342212866627017160926156555122732057869537884877103875740958970120407365354535}\\
37, 53 & \seqsplit{14409360676462526169460275751372545013613431985219209755871179086505525643222559186964}\\
38, 52 & \seqsplit{35608755592110430446729481821941989073626495014691830681322873962264428268042303559920}\\
39, 51 & \seqsplit{77788303937603800024393018700115795383182690042352024704097745705097268010612414468417}\\
40, 50 & \seqsplit{150392283185745672239827941221876277052458954569160210873241409905740212957164128935785}\\
41, 49 & \seqsplit{257581144928353852713586500950687585156524349521553276899281195468028469601078659944090}\\
42, 48 & \seqsplit{391130748777193293204548248616898516104384453242214881989099810196146488368944292249514}\\
43, 47 & \seqsplit{526881082344555432681735679683330830580530067893167647631157346604454090967179528021070}\\
44, 46 & \seqsplit{629901596509567178000198917231743704503446344288732507018644083610752603338795547000795}\\
45 & \seqsplit{668518953150772585395358734679378019961377745979447877333202919455961351954200317292072}\\
\end{longtable}
\endgroup

\clearpage
\begingroup
\small
\setlength{\tabcolsep}{5pt}
\renewcommand{\arraystretch}{1.05}
\begin{longtable}{@{}>{\raggedleft\arraybackslash}p{0.075\textwidth}>{\ttfamily\fontsize{7.6}{10}\selectfont\raggedright\arraybackslash}p{\dimexpr0.925\textwidth-2\tabcolsep\relax}@{}}
\caption{Ehrhart series of $\B_{12}$: $\Ehr_{\B_{12}}(z)=h_{12}^*(z)/(1-z)^{122}$ with $\deg_z h_{12}^*(z)=110$.}\label{tab:series12}\\
\toprule $j$ & \multicolumn{1}{l}{$a_{12,j}$}\\\midrule\endfirsthead
\multicolumn{2}{l}{\small Table \thetable\ continued ($n=12$).}\\
\toprule $j$ & \multicolumn{1}{l}{$a_{12,j}$}\\\midrule\endhead
\midrule\multicolumn{2}{r}{\small Continued on the next page.}\\\endfoot
\bottomrule\endlastfoot
0, 110 & \seqsplit{1}\\
1, 109 & \seqsplit{479001478}\\
2, 108 & \seqsplit{62347317909591781}\\
3, 107 & \seqsplit{569053303915793430407960}\\
4, 106 & \seqsplit{683258212723871263939221247640}\\
5, 105 & \seqsplit{172767090726819833331709826567540408}\\
6, 104 & \seqsplit{12731958819430221293010423384404665815120}\\
7, 103 & \seqsplit{344367811477601237943490283674673169826229736}\\
8, 102 & \seqsplit{4048231819081906569596317151797519834709198068356}\\
9, 101 & \seqsplit{23496770241107320996113439775124088789431980394210480}\\
10, 100 & \seqsplit{74306538292562215274099629598392596472035911668548117396}\\
11, 99 & \seqsplit{138363221143506409308468497590757216504477021121620329364040}\\
12, 98 & \seqsplit{161431808000873242335976749449177536454387583370200703565860093}\\
13, 97 & \seqsplit{124132365396179335775957990714154666475155681513401559381393099430}\\
14, 96 & \seqsplit{65582701492568392730203748623517276112389128316487451240579252408801}\\
15, 95 & \seqsplit{24647107144948393234573514831002084665403848680592909985657841744437088}\\
16, 94 & \seqsplit{6784149022466451752722943259400601414995292722071426962310244328651962154}\\
17, 93 & \seqsplit{1402009308704257034170220742955625283474617771724311282600482862687827802012}\\
18, 92 & \seqsplit{222209453547718201931834416512189710165365234548824227919719505467676013707298}\\
19, 91 & \seqsplit{27510488274362394043541723553983809941133740700831166468933481611350870057270512}\\
20, 90 & \seqsplit{2703237804409676562217433591080413083954434952555289578120141112710043161665311060}\\
21, 89 & \seqsplit{213785611283196211189318352509648921787559503728154245714611976407386330598233600424}\\
22, 88 & \seqsplit{13775583400282562789051106600572908164985303340068251054884894814214611985290664850244}\\
23, 87 & \seqsplit{731123615385071167301997139733655460687554094643384026189009673626477566952143131216880}\\
24, 86 & \seqsplit{32270583052379290747650605956767935134656762383008641456626096583651069911512972289174332}\\
25, 85 & \seqsplit{1194792219258657650086905299294525515244158787024574804299175074617188522502952554871987608}\\
26, 84 & \seqsplit{37393138237235299006837502299813367169689832392846832876997810524557638744900355681578115020}\\
27, 83 & \seqsplit{996118251422698202778332286527993210299998487228482980541709398253995284559396899695628419376}\\
28, 82 & \seqsplit{22727867671408545930367241313475027629464589413911139775657781295254513373142717714823064238650}\\
29, 81 & \seqsplit{446667629211193033697395974353684128278653768582438439833379001762454889228495570664074783918092}\\
30, 80 & \seqsplit{7599896576423069773276825340742214761498072808057536257530032874431841969351855627942789012323714}\\
31, 79 & \seqsplit{112471847914390014991390044144872855555471679964955253686992432422304413711309805746791488635524640}\\
32, 78 & \seqsplit{1453873643958820096553288536016994606092926610390209433273885398539284893539960191206027156327713879}\\
33, 77 & \seqsplit{16478857359693429382861738982556713317301969767396409326140565143063816937605213394810139216603312778}\\
34, 76 & \seqsplit{164349941831706765817325685888494001077050388058181907165678399860789336884901372387139500330040668587}\\
35, 75 & \seqsplit{1446925365491579480158006826024582981662467796587796225294973924826665036673464177005076079428987699544}\\
36, 74 & \seqsplit{11277915337439814696790046096735596901096954197905041496230614342600285163814252992078512586133363229260}\\
37, 73 & \seqsplit{78033352779241712599644482146663557735892553947711272547553095934297771250269857221237068186956751839440}\\
38, 72 & \seqsplit{480466354446994257384798802321501978918104498716625576453208319715765362510407979964556601695801833835308}\\
39, 71 & \seqsplit{2638442248033301445514495682222464958184649897032146611219801581043473178388360605003689977281100532375864}\\
40, 70 & \seqsplit{12948514746510776020799030992771776377285218403180525119064326349772255777652425141509174515077154146634416}\\
41, 69 & \seqsplit{56896746237789677361944350302052177907632387997694106431842291526093818298955939135621251576589061068222152}\\
42, 68 & \seqsplit{224225189538399668635125817256693419441856516137626074020691485425441602477492411594226788580491249235310200}\\
43, 67 & \seqsplit{793737885780515555291238311462141124693048831196960379126742050523668298748097141704061163291733406510033160}\\
44, 66 & \seqsplit{2527364810486370502992502754496944607161143863751953091157427451951520372339875532251106513445143265933335907}\\
45, 65 & \seqsplit{7247681608975837771830869299080344565293237305798769052514975049236334263013041136549862357607746096992265546}\\
46, 64 & \seqsplit{18739413495661474678774060305712181476677476460141767259866462999081174678800843862113396721565118055021622519}\\
47, 63 & \seqsplit{43729041266969727420763515291244580391942947490374861836238562657851596757734322755256743877873353173271288896}\\
48, 62 & \seqsplit{92176628943222106431074506306731576003789013305435857064577632731664920416734059333093873664087634939548164004}\\
49, 61 & \seqsplit{175646310869267756178471264369856414930038034443027462917878764993477682742005858271438483528742919600608754968}\\
50, 60 & \seqsplit{302766257775845397168777121266980968428812574758637343562884925342491492917166204378273704902468376216198014388}\\
51, 59 & \seqsplit{472349005721892088322517419022112735351258736223094254779262551016161135530347147250271034163110635934659679648}\\
52, 58 & \seqsplit{667264080031427793513531200129421874317953624655732527609008346657892912466762408758501698641736433816972706032}\\
53, 57 & \seqsplit{853807736129043857131860650568726406411424655745357949086086526444538467598736475576247856495164964998741508864}\\
54, 56 & \seqsplit{989818267671217012511142208727752639825515043414995237987645318939624490308721282987386447691024940783690592144}\\
55 & \seqsplit{1039793100349881521714365926382488627313133940754037744599895598314911519343746261480704282660342205460981049696}\\
\end{longtable}
\endgroup

\clearpage
\begingroup
\small
\setlength{\tabcolsep}{5pt}
\renewcommand{\arraystretch}{1.05}
\begin{longtable}{@{}>{\raggedleft\arraybackslash}p{0.075\textwidth}>{\ttfamily\fontsize{7.6}{10}\selectfont\raggedright\arraybackslash}p{\dimexpr0.925\textwidth-2\tabcolsep\relax}@{}}
\caption{Ehrhart series of $\B_{13}$: $\Ehr_{\B_{13}}(z)=h_{13}^*(z)/(1-z)^{145}$ with $\deg_z h_{13}^*(z)=132$.}\label{tab:series13}\\
\toprule $j$ & \multicolumn{1}{l}{$a_{13,j}$}\\\midrule\endfirsthead
\multicolumn{2}{l}{\small Table \thetable\ continued ($n=13$).}\\
\toprule $j$ & \multicolumn{1}{l}{$a_{13,j}$}\\\midrule\endhead
\midrule\multicolumn{2}{r}{\small Continued on the next page.}\\\endfoot
\bottomrule\endlastfoot
0, 132 & \seqsplit{1}\\
1, 131 & \seqsplit{6227020655}\\
2, 130 & \seqsplit{10112898638215583640}\\
3, 129 & \seqsplit{868070264782555036647948360}\\
4, 128 & \seqsplit{7821494250389134849100378043130470}\\
5, 127 & \seqsplit{12389096268335235606384110924372784224058}\\
6, 126 & \seqsplit{4933926897085751430481500879432659720658062904}\\
7, 125 & \seqsplit{637645691434988720431515477000572788211602501294920}\\
8, 124 & \seqsplit{32279622996433463906836079351586608271460544797411427475}\\
9, 123 & \seqsplit{738364749562755992788204570583832542823729184547702225968645}\\
10, 122 & \seqsplit{8527505425933609412352516360959632586668967037230661975681267548}\\
11, 121 & \seqsplit{54300317971878401357013068342613043642418224056660292477728311514972}\\
12, 120 & \seqsplit{204651123247514213130169661727371906174387678728062848953729347075617363}\\
13, 119 & \seqsplit{483763046806824322211067381245644635374844444020226442858361385954540426290}\\
14, 118 & \seqsplit{752467823203247452231654527325207560634581238206643160563594738553153891468195}\\
15, 117 & \seqsplit{801680536103942014611963817774358631528664941588402480271000454861704729324255200}\\
16, 116 & \seqsplit{605180380106926219104587539009670413977332036687315536461577870392970900365084560754}\\
17, 115 & \seqsplit{333173672495374147540387510770402481610214797278123384570205225461889319070042518918786}\\
18, 114 & \seqsplit{137125716973378084417760480628809171403266018967989389479846585769166382611574441113154880}\\
19, 113 & \seqsplit{43105710134377978068486560252676013554592261311972621040892023154962185589085006585679211800}\\
20, 112 & \seqsplit{10544324547183528291155797462597503473284541809481026817914236666105163392915094361909089763762}\\
21, 111 & \seqsplit{2040164642548466156144142606695898007567817602695155762852377917321381157070624854121637042223279}\\
22, 110 & \seqsplit{316754121706692014978044906531545497592995121617328683125094234697123863683076752522759273741218855}\\
23, 109 & \seqsplit{39968794961282573523922065603225261817690835403943165807354080317224333999724429173434445176579996860}\\
24, 108 & \seqsplit{4145541033235784270738911805794949809638634735322278282044020826243398808080289535903838579701961435955}\\
25, 107 & \seqsplit{357024283484786238495858480211619524790284309656305199110954375682644402907169093755957118115652288500291}\\
26, 106 & \seqsplit{25763973067773484786829153067363009642878952861209809978920092079493241674608967071667486978787487275427188}\\
27, 105 & \seqsplit{1570636444630596916496535881920676553263565058141703491950711584442483738185004853673346309558830440008304106}\\
28, 104 & \seqsplit{81487820131389805326592451246577835854722714184893520021888451011145355539392089986393704282839115921540780215}\\
29, 103 & \seqsplit{3622179016538247491107747147280257229136256771777464021440364638976536186699819457251669862102674721822767579400}\\
30, 102 & \seqsplit{138786969700553096823341115805031381483835947199980322152260336750740971446744937839037995337784672271661051314971}\\
31, 101 & \seqsplit{4609315165166869526194012143425150351183982366785287187564921798940828457190987763818216984684034723186976721011634}\\
32, 100 & \seqsplit{133362127693228400096769737220653282221502957287170784042581249085484239807688018233443615001059854537699664626852592}\\
33, 99 & \seqsplit{3377155888145498406170451087592417037557908930967454906654997676002671728429609976042933089584239825146989497675123685}\\
34, 98 & \seqsplit{75169262684122665237696151574558127807922841244606336127566742338137897674396995390901231987469126135936149635552874680}\\
35, 97 & \seqsplit{1476383012995285087579717470255928147647771437529793739953625319161876597587800904430952151197043449905542211876429846565}\\
36, 96 & \seqsplit{25679812819999573938672552228398030705000065067907352216518473595955998489410986622780082265940007914929910324969479852303}\\
37, 95 & \seqsplit{396882136078869155798977033409799250100972932500282577097672007181376930933598737076435921045843099788830380485133834417814}\\
38, 94 & \seqsplit{5466891909354705965337671585570624140832108066486666371796067045879830940307007983085598451540471103772288353018565964869125}\\
39, 93 & \seqsplit{67307027494816819274532120827159494618945123469199786797342113249230625491037609918908890720045027748782402367635249354087060}\\
40, 92 & \seqsplit{742610629496465486190629795345753352025299425874799092468604703818056064250935332042950803162160834097417694109046886213157376}\\
41, 91 & \seqsplit{7360347998426407847578703752723134719728206937595985425818404140183634310378907476849197277754472182595618399959742209554748792}\\
42, 90 & \seqsplit{65682574293988103176308762252764930105438998950509773780233487044398919165424463518981966703258479180842939644457033894598882644}\\
43, 89 & \seqsplit{528838957060121988927011891852818932898316950475580066946846064506844947184978692446561850197299239604318398625728352357589574505}\\
44, 88 & \seqsplit{3849098428301764726895411714129122036858158598638652586383889454798349475730222337412880940863946029127106876011615951577267323190}\\
45, 87 & \seqsplit{25370940361461493906830891962259316466607590806268901610348611503712945741577889560956765025984243022759388212332141204851424020736}\\
46, 86 & \seqsplit{151697654167198101705382645181035406581101658565043297479051773398729015268832632313427904039081878122851137110229553516767902201549}\\
47, 85 & \seqsplit{824053967471368205272824425702709409353070387154669470998891600371579374809299086945311835025924746125262105101318545496773966761393}\\
48, 84 & \seqsplit{4072743533081877439410889419252219215439778945104440596169749612405228507312206148198373820055574253034794035572912465389013252401165}\\
49, 83 & \seqsplit{18337777333594623710939815316387131459432813001874040870767262800633200655074131557479894588351020808912596780790032503463467069554940}\\
50, 82 & \seqsplit{75312037967913466303637406218232508015889730256260644580743041815729054826851112714924160638988055653435819031266056697549449541276202}\\
51, 81 & \seqsplit{282441657060394784426119820836067217187099131745372682673902709124083276471084564607452910087390687632965132016346715859002470309638738}\\
52, 80 & \seqsplit{968253939849616952661937641720408358888467698543983340804011078348124195969365844840806197032913982070824067430739365880019989565585355}\\
53, 79 & \seqsplit{3037095440350696503260516599354329384885373509567544845540316474986765829953796768871730338700571939052001890658401033639627912377566120}\\
54, 78 & \seqsplit{8723955924405077126355142927563141383364788939691354608860761660222218793028111720365640114784550838903524253040059581428745852389455825}\\
55, 77 & \seqsplit{22966662541941746185438397234478376735398053049873178615167548295716809766373118187495108588390761165792623454860772811146907313706188290}\\
56, 76 & \seqsplit{55452742449355171344401255344586348448486563342960758116262814709139795291579276066890188146354645833160619863381860438054263559539774832}\\
57, 75 & \seqsplit{122876610715565084678233046102415597454522327648687992878786361509359987320098460766047795941442157650944854176974511252796149656678821624}\\
58, 74 & \seqsplit{250027031840722895511725431260071712551014927526124207933166098811049235524903513407414351741550667862363135179187830444849748353735640460}\\
59, 73 & \seqsplit{467409876057027954045950422426009166208049238348340691868206500295807150775751718305155736963640324581550758465367975029489167274781206365}\\
60, 72 & \seqsplit{803148052092515823071456517549329417084905532699028989650377760551476367664666370592993966412559067402847956961461860066813479686218411514}\\
61, 71 & \seqsplit{1268955847745749233898325214218630020051189700724810541630764962800378959872211884211189810132102225497590087446689732115062772146776439044}\\
62, 70 & \seqsplit{1844124875230661174672637702995920084563958869092487137784568511800048828725983315792858107811050824647435323033300498789238903365448234300}\\
63, 69 & \seqsplit{2465702562128595702891696400232475478041624081002245228748605632322900710540891671869827042414937035272577119753052352072554736392073480290}\\
64, 68 & \seqsplit{3033790182969804602470741296499433197382079203768730413863826873485224028263694908903419511990623013941092205732949223141639061486027663610}\\
65, 67 & \seqsplit{3435480458445457848798027735118512900693125304217870053709078552420056949221785690276794463017090360222616298767755926116705190853422939099}\\
66 & \seqsplit{3580832215256516940245227550503809003287250476127053521604805254192508621589186336071103828538654686309920601582247446978844816464283035904}\\
\end{longtable}
\endgroup

\clearpage
\begingroup
\small
\setlength{\tabcolsep}{5pt}
\renewcommand{\arraystretch}{1.05}
\begin{longtable}{@{}>{\raggedleft\arraybackslash}p{0.075\textwidth}>{\ttfamily\fontsize{7.6}{10}\selectfont\raggedright\arraybackslash}p{\dimexpr0.925\textwidth-2\tabcolsep\relax}@{}}
\caption{Ehrhart series of $\B_{14}$: $\Ehr_{\B_{14}}(z)=h_{14}^*(z)/(1-z)^{170}$ with $\deg_z h_{14}^*(z)=156$.}\label{tab:series14}\\
\toprule $j$ & \multicolumn{1}{l}{$a_{14,j}$}\\\midrule\endfirsthead
\multicolumn{2}{l}{\small Table \thetable\ continued ($n=14$).}\\
\toprule $j$ & \multicolumn{1}{l}{$a_{14,j}$}\\\midrule\endhead
\midrule\multicolumn{2}{r}{\small Continued on the next page.}\\\endfoot
\bottomrule\endlastfoot
0, 156 & \seqsplit{1}\\
1, 155 & \seqsplit{87178291030}\\
2, 154 & \seqsplit{1908371349022450726765}\\
3, 153 & \seqsplit{1663043403250261843934199873160}\\
4, 152 & \seqsplit{121226473691251044312171538524818758720}\\
5, 151 & \seqsplit{1294607948007407200328006215209076536524957368}\\
6, 150 & \seqsplit{2992325985654138641471064325718029858429706269462792}\\
7, 149 & \seqsplit{1979389548714287722886998238157527935204213120819154294584}\\
8, 148 & \seqsplit{460805594834524201380941879424339822372262062209263312705635295}\\
9, 147 & \seqsplit{44200235420605646029251428721679952342479163356648847395470254526330}\\
10, 146 & \seqsplit{1975633280461629867833177099709424115261343502791127877515752748700301269}\\
11, 145 & \seqsplit{45388251596356849591406050166832563124746954505627014052181875660015537405476}\\
12, 144 & \seqsplit{580258642232606460002051023888055773029063182437491159018476186903435972721663454}\\
13, 143 & \seqsplit{4406199882672766718715357264066582500558466596595748125465634788268463256996145776920}\\
14, 142 & \seqsplit{20980413483125673399675046280359301725992186488674233742141901199222371337687739843289651}\\
15, 141 & \seqsplit{65561012771529696030846152532933603038610345502026064485996192700952747838380986258946563466}\\
16, 140 & \seqsplit{139740411379033927239456566380564353883119194555852368344911752300166834237814053051390269724456}\\
17, 139 & \seqsplit{209976508980520996161352607026852238840273875969222526393247908853246545900492122616738466688029614}\\
18, 138 & \seqsplit{228838510931556587140931343939199972616452483194007100378383539453349965701192509903818641808397631849}\\
19, 137 & \seqsplit{185392876427255757227604529080832828426672708132068724943887934767377147311235583087125310913158959194152}\\
20, 136 & \seqsplit{114075645241459039057434996382525194347583201363662528668986644301519441623749989956965113799087869737104975}\\
21, 135 & \seqsplit{54326890564704446167435712416718074955257210925234284342522670347624874548114201071835820695425466233015499594}\\
22, 134 & \seqsplit{20360055703412540672410248497455051357497447386299859956986215164700202477212572262314923762172124577117357736872}\\
23, 133 & \seqsplit{6093753688849395950411430197702175047670475519200038805372757656512603757837764827417227245300800698838236743162694}\\
24, 132 & \seqsplit{1475811716154552738998509764796575349249619804764019660764454635150907017029130904273234390531404106079881166586832570}\\
25, 131 & \seqsplit{292625804112191038702862458434663678207423088930970889644847444764350125930207231110178449503397597810618507041073170438}\\
26, 130 & \seqsplit{48006988196917258994026696254486666859534508144536751858289941945440617335032716974624633866910250942951145720718641565977}\\
27, 129 & \seqsplit{6578526493669926039153455250441653421898245946555066281545870842155009848678190889175660804879052243402452014650591998715104}\\
28, 128 & \seqsplit{759473264059517999922611627717956037509399953979456501397808235691376572139487607568135817455270843570572849366658816351267301}\\
29, 127 & \seqsplit{74445313377565733311505419308073919624709918825623409624415065128448720079940809316286922465485492113762854780072166589973395022}\\
30, 126 & \seqsplit{6239933081678406194252876826181680548752220879388750417996107623943989967001558906834859935533015170253296563698357610166575395688}\\
31, 125 & \seqsplit{450140023442610677038024685290326699333134334928464300155150262041842460207141439770164498610976404973750753288566349154922307611298}\\
32, 124 & \seqsplit{28113115119151238647492587096276211189732747633468563249950766021241240595482332957296928289146566536484639233043370531536148465225504}\\
33, 123 & \seqsplit{1528339034349148189176732722434408550657344768428027803667991864111106985574877695440565227537789662090119104101409822224963535637498670}\\
34, 122 & \seqsplit{72685105320277350120234337578410873991999562398940846132817173182226373742967744457911726529809408785286257904859109703452365828793203816}\\
35, 121 & \seqsplit{3037954051015260172227220886335972170468517833071837757920470799693839023387722286497180048436604545087384227554049539873106595590541155282}\\
36, 120 & \seqsplit{112064667521382372465791674940369669839650977849172636872936930544904487584717081370050128757660681811049768404971563325168405265216429886502}\\
37, 119 & \seqsplit{3662782931009438714357213953134929773253664343790385973844412476722820772013896693860393140761096285629179536016677608839978286198670415717922}\\
38, 118 & \seqsplit{106460026500338507680071413005096045649753332571957778800819572589945876995346151012464141190901210437434924672555985080750830252724904825419864}\\
39, 117 & \seqsplit{2760953476592831006900012414750604308202293777647694619565698524852155311581723543957603977314282737773018568853042778747940999576726721960872382}\\
40, 116 & \seqsplit{64089877206999126159594568152913668326666940780470883798809163534829849374653868565651268935153623499274790173454274917519177648188045458899141232}\\
41, 115 & \seqsplit{1335499882171430259659989129371930160644464560178781793845278762218324116123468368352744270677347342510459242345068148783065256203892902877304127458}\\
42, 114 & \seqsplit{25049676711934747844393477780247161001240106116467978562141299455286156608261206820500226973491501233907176549242222911314910064212814925436824905492}\\
43, 113 & \seqsplit{423998928257169627988952811602340390771024540141821441451234872164290702286662093621898651938260200917746837012451218856102118176134153475857009217126}\\
44, 112 & \seqsplit{6491713657569390362490185010625210294408481438138433034369535119428748591461551617938330762675480023490719955927731292943703376292964169625898147881374}\\
45, 111 & \seqsplit{90104466684995914452843580995847813411283991336461777842068427025767463180555484043783747295583653531550279726081385741660400577334326890000301505014990}\\
46, 110 & \seqsplit{1136119423507780986307736844939621446641679625153650552663867402966829906838623327714540545912564552673990899941007412542997484782247668002041117619985812}\\
47, 109 & \seqsplit{13038667282260268399390717619025998987815735981785019109801053760281680180481743357678596203914228555917996390747909191805247185392209270854400011187198650}\\
48, 108 & \seqsplit{136446216245492753908925198532836824154404997946699013333783118040940634735646735349519188546389234044854170373901046056802645943674629948234161331035648358}\\
49, 107 & \seqsplit{1304207782492098044874048254056353136377155153364992513522724114937224981335845883203794716525841850343620548804528921272612586479203168432171568241615479154}\\
50, 106 & \seqsplit{11404618759024834102974353023349509433386988975558612794207296046962764434474849983082018683176600864354769311332583427418910628393923194647279455201401744596}\\
51, 105 & \seqsplit{91371757114558523047905972248309890159753166254369846990228462773990046295979808720807407541847233935773159190457534040369352728837383790078862790821477894758}\\
52, 104 & \seqsplit{671658513228829993877561168289379230095658251669282985069985680674048526273712317087191372935149234726433088936505296161245847801923858966643559510275049933460}\\
53, 103 & \seqsplit{4535870462284523452511919919435855963606531636944975406060287222735620501197356249024808352018687768428551563925626764944048794636952863849814344927967940739706}\\
54, 102 & \seqsplit{28176234559694993738763378435584402320461343547914089651687625623115796853945448759665702366564208570916636219701870202685200214902532236184157421312707169162944}\\
55, 101 & \seqsplit{161181999533915807426938517582382715895697743966505269476213188689956817822237592577904777753995369476263907721146089142010209076536806317509536832586544689328790}\\
56, 100 & \seqsplit{850023307979998462982062615746653418794202211861617635678127507887441284843211758139414948905330102562908825751072178010394687124239463984360951546262130818185222}\\
57, 99 & \seqsplit{4136807857158829279970956458173013012105195405891563596118535436719370691599650588636074140767863598994462761630929990171716017795917443741789185909471969918948646}\\
58, 98 & \seqsplit{18596500941468740010478870704241372865158507541809275656001514314084716739347654789352811012396285852012229086046582869697069806415140778687443658850060248804740016}\\
59, 97 & \seqsplit{77288115569823548772698422974140709885637384142087294073855639892788872082585255596001053317989268468389738114823883706170743111592854482489215780831683211829961946}\\
60, 96 & \seqsplit{297213406633086403304821162575653139469225009739342601984323029240634614572299833866379704044759420239105008933692018743883695161738448810428184298586119613379629420}\\
61, 95 & \seqsplit{1058357599132934784021045644688149732862601207226530901797233362112969885010822357389118779148456366641354252133035878025536502065048229023180479812125345753626645382}\\
62, 94 & \seqsplit{3492332958417050241675351012395113769668689000760342705483671598951551813918967018906914048027915933940029759251225305483924159815462410118466194108671975612439227216}\\
63, 93 & \seqsplit{10685786068654685161530666881036801995577341065125742994676464827283323254227663810731884541030774807734208743828731878900213134717285014109407506421636457875483435322}\\
64, 92 & \seqsplit{30336949854848287977730040214042130793551849203852012309842266952272352085541119342514139432318097996747524012273604316028522541818777141420681675683740265720809463006}\\
65, 91 & \seqsplit{79957331337496179259559833277744806609493823513251034031881765589294872679047317498578521446698238210739894019075699102254666752202059643824107642689568219883849041714}\\
66, 90 & \seqsplit{195745608821123503749323081575773405638855993809542858619222519595711002679821087157554792134654894562102621652708987091604650470340679545170264685879142942051004168322}\\
67, 89 & \seqsplit{445328312523122682668041167120149952747017912394098149615681269734364751278209973321551121283139358887028288322192126391994712207124793154044342027204935677430388660810}\\
68, 88 & \seqsplit{941914903199724345688705932937341782336887898492374715108936093419226737288438738353884893644298625373120306772508542497020694407281928943658228757451636508948905637242}\\
69, 87 & \seqsplit{1852919518359412901749477570493037240256704668418225223683510546047582869821758250648986384263861788426314878437534257020255843353118152268863265730365554601406389331066}\\
70, 86 & \seqsplit{3391312404965691447380056189538692848201888341232801802960461168847254633997235725847978963431930570476068279343666712269829782111500418074490780382614387016439412084472}\\
71, 85 & \seqsplit{5776716263181014877474203027853744957073397638621980080686340468147334428169700620362662429791619080944957123368352732415588231724204553697377120810622337404996855963670}\\
72, 84 & \seqsplit{9160413970907649152134854115291518091487383053711320230557933544905536109789979996919887062177510857933009523670449659306746786694130707894332441090643355142135962998478}\\
73, 83 & \seqsplit{13526091908836518922382684922219377424132094942387296569752239149016354570773194668771476475411108769503780547993159160319732541755522512333970807629399022876688629865574}\\
74, 82 & \seqsplit{18601110611927685399690997607602997477736601660321164531876678349825923146004989596331914012398562125171723080716567289358253946436941659696438634888737604693059212568826}\\
75, 81 & \seqsplit{23827851664752697129407532421064392482441974069014622717943876247711235485992681990375466532389974162422462428507805694833869051099512901312972799964198230858378025234950}\\
76, 80 & \seqsplit{28435754109749175557452705990469215813179874898429345649944522853205611983822552431604260648284726723151394190641917151598581943370097979592047503527456761127486051785830}\\
77, 79 & \seqsplit{31616748227959937236467900366806929427744633056933422125814687418453504937492362122080215859777074415135718693177191440551133668122830417411076973507628667460722998709606}\\
78 & \seqsplit{32754082760354682124068599281119050123377646820588114804605887557001261514064378105065828772901707067504860546098097644650639374608118710205587006730328064270469934086338}\\
\end{longtable}
\endgroup

\clearpage
\section{Ehrhart series for the World Cup problem}\label{app:worldcup-series}

Tables~\ref{tab:wc9}--\ref{tab:wc13} give the complete numerators
for $9\leq n\leq13$ in
\[
 \Ehr_{\mathcal W_n}(z)=\frac{g_n(z)}{(1-z)^{n^2-3n+2}},\qquad
 g_n(z)=\sum_{j=0}^{(n-1)(n-3)}b_{n,j}z^j.
\]
As in Appendix~\ref{app:birkhoff-series}, two indices in one row specify
the common coefficient at both positions. Every numerator coefficient
is included, using $b_{n,j}=b_{n,(n-1)(n-3)-j}$.


\begingroup
\small
\setlength{\tabcolsep}{5pt}
\renewcommand{\arraystretch}{1.05}
\begin{longtable}{@{}>{\raggedleft\arraybackslash}p{0.075\textwidth}>{\ttfamily\fontsize{9}{10}\selectfont\raggedright\arraybackslash}p{\dimexpr0.925\textwidth-2\tabcolsep\relax}@{}}
\caption{Ehrhart series of $\mathcal W_{9}$: $\Ehr_{\mathcal W_{9}}(z)=g_{9}(z)/(1-z)^{56}$ with $\deg_z g_{9}(z)=48$.}\label{tab:wc9}\\
\toprule $j$ & \multicolumn{1}{l}{$b_{9,j}$}\\\midrule\endfirsthead
\multicolumn{2}{l}{\small Table \thetable\ continued ($n=9$).}\\
\toprule $j$ & \multicolumn{1}{l}{$b_{9,j}$}\\\midrule\endhead
\midrule\multicolumn{2}{r}{\small Continued on the next page.}\\\endfoot
\bottomrule\endlastfoot
0, 48 & \seqsplit{1}\\
1, 47 & \seqsplit{133440}\\
2, 46 & \seqsplit{5655744444}\\
3, 45 & \seqsplit{42816738899480}\\
4, 44 & \seqsplit{88847153335965150}\\
5, 43 & \seqsplit{68816208438242976312}\\
6, 42 & \seqsplit{24406702066529359243044}\\
7, 41 & \seqsplit{4555822450138168507665192}\\
8, 40 & \seqsplit{493914660218684937648064035}\\
9, 39 & \seqsplit{33426827860740619521317973976}\\
10, 38 & \seqsplit{1491104627315096991852041726856}\\
11, 37 & \seqsplit{45720578475824106766297718520048}\\
12, 36 & \seqsplit{995954935039493043448106089815446}\\
13, 35 & \seqsplit{15824922452861009878253107825728336}\\
14, 34 & \seqsplit{187357729913568393455469773498820552}\\
15, 33 & \seqsplit{1681796954331440355889353200942754048}\\
16, 32 & \seqsplit{11609930770614339123041749640294775189}\\
17, 31 & \seqsplit{62361382098118337873221967549454367008}\\
18, 30 & \seqsplit{263142474644251013358470428347461090464}\\
19, 29 & \seqsplit{879102517773995687301147713692570144704}\\
20, 28 & \seqsplit{2339794100701495050904979647518857037348}\\
21, 27 & \seqsplit{4985786705356378412445782387668813140680}\\
22, 26 & \seqsplit{8537047697622791855580797428396999207020}\\
23, 25 & \seqsplit{11776441703713092363361692904046401608792}\\
24 & \seqsplit{13107214837161535660241329225554910063065}\\
\end{longtable}
\endgroup

\clearpage

\begingroup
\small
\setlength{\tabcolsep}{5pt}
\renewcommand{\arraystretch}{1.05}
\begin{longtable}{@{}>{\raggedleft\arraybackslash}p{0.075\textwidth}>{\ttfamily\fontsize{8}{10}\selectfont\raggedright\arraybackslash}p{\dimexpr0.925\textwidth-2\tabcolsep\relax}@{}}
\caption{Ehrhart series of $\mathcal W_{10}$: $\Ehr_{\mathcal W_{10}}(z)=g_{10}(z)/(1-z)^{72}$ with $\deg_z g_{10}(z)=63$.}\label{tab:wc10}\\
\toprule $j$ & \multicolumn{1}{l}{$b_{10,j}$}\\\midrule\endfirsthead
\multicolumn{2}{l}{\small Table \thetable\ continued ($n=10$).}\\
\toprule $j$ & \multicolumn{1}{l}{$b_{10,j}$}\\\midrule\endhead
\midrule\multicolumn{2}{r}{\small Continued on the next page.}\\\endfoot
\bottomrule\endlastfoot
0, 63 & \seqsplit{1}\\
1, 62 & \seqsplit{1334889}\\
2, 61 & \seqsplit{535206884364}\\
3, 60 & \seqsplit{28876886666493276}\\
4, 59 & \seqsplit{344107119241648085400}\\
5, 58 & \seqsplit{1302633907536847863868128}\\
6, 57 & \seqsplit{2000616755057074114847528292}\\
7, 56 & \seqsplit{1475829111622183900133270258412}\\
8, 55 & \seqsplit{589916676350854205085786256430568}\\
9, 54 & \seqsplit{139636462560929573926157913690586990}\\
10, 53 & \seqsplit{20934596917433719342802221849463198481}\\
11, 52 & \seqsplit{2094224945528823519592202387816703244599}\\
12, 51 & \seqsplit{145671020122287495920523579618493190019744}\\
13, 50 & \seqsplit{7282983473205583087823904896052587472210096}\\
14, 49 & \seqsplit{268892263117669888150025370982438110559616430}\\
15, 48 & \seqsplit{7496926592348282964375796370659481179203937530}\\
16, 47 & \seqsplit{160813419256063354549027416238342664997938264160}\\
17, 46 & \seqsplit{2695950069897106925070930717288344447989761736380}\\
18, 45 & \seqsplit{35795580980215373850047742209990612480736271434610}\\
19, 44 & \seqsplit{380713566907256609547468073821996843119685867187980}\\
20, 43 & \seqsplit{3275145053154167512790901630174266323856297512662255}\\
21, 42 & \seqsplit{22979349657809035545464291447106931887067503821108925}\\
22, 41 & \seqsplit{132440049141829206738139193202393936552612584475629760}\\
23, 40 & \seqsplit{630858673405885676304913427906610880868422369544694240}\\
24, 39 & \seqsplit{2496581467009851104193392816593829207075570591550298980}\\
25, 38 & \seqsplit{8244945375870403285783479936631701156637555346850726888}\\
26, 37 & \seqsplit{22807577911803234487085819440808009643465312253223239602}\\
27, 36 & \seqsplit{53010637967898902661673489377002585357330664970483638292}\\
28, 35 & \seqsplit{103783285852130288283331698089972537151680338972631551848}\\
29, 34 & \seqsplit{171483773505699674589806699634874499006511150908933108940}\\
30, 33 & \seqsplit{239484119932406838939286545545633243195587711062210728949}\\
31, 32 & \seqsplit{282944441357126351563193768813627714528025427634409244311}\\
\end{longtable}
\endgroup

\clearpage

\begingroup
\small
\setlength{\tabcolsep}{5pt}
\renewcommand{\arraystretch}{1.05}
\begin{longtable}{@{}>{\raggedleft\arraybackslash}p{0.075\textwidth}>{\ttfamily\fontsize{8}{10}\selectfont\raggedright\arraybackslash}p{\dimexpr0.925\textwidth-2\tabcolsep\relax}@{}}
\caption{Ehrhart series of $\mathcal W_{11}$: $\Ehr_{\mathcal W_{11}}(z)=g_{11}(z)/(1-z)^{90}$ with $\deg_z g_{11}(z)=80$.}\label{tab:wc11}\\
\toprule $j$ & \multicolumn{1}{l}{$b_{11,j}$}\\\midrule\endfirsthead
\multicolumn{2}{l}{\small Table \thetable\ continued ($n=11$).}\\
\toprule $j$ & \multicolumn{1}{l}{$b_{11,j}$}\\\midrule\endhead
\midrule\multicolumn{2}{r}{\small Continued on the next page.}\\\endfoot
\bottomrule\endlastfoot
0, 80 & \seqsplit{1}\\
1, 79 & \seqsplit{14684480}\\
2, 78 & \seqsplit{61586749444905}\\
3, 77 & \seqsplit{26090433141795793770}\\
4, 76 & \seqsplit{1949896245985521637561890}\\
5, 75 & \seqsplit{38887554637493025533281581732}\\
6, 74 & \seqsplit{274436503967332864935056826776130}\\
7, 73 & \seqsplit{834861204051116539791552716647238040}\\
8, 72 & \seqsplit{1262468565874275814061849521185317974800}\\
9, 71 & \seqsplit{1055034049001730628332007639280470697326590}\\
10, 70 & \seqsplit{528220220567265489589490586391311782932919933}\\
11, 69 & \seqsplit{168722812019945291301663398313336771136039303640}\\
12, 68 & \seqsplit{36142153752964941224824475915989071127997115829550}\\
13, 67 & \seqsplit{5405282977197910902423355778290950556889023570622070}\\
14, 66 & \seqsplit{583303887944494651882651599116112229110016665482309145}\\
15, 65 & \seqsplit{46676823741787938452355979219141160917959419563398318924}\\
16, 64 & \seqsplit{2833834781963783363511641891091645486133705302632922352995}\\
17, 63 & \seqsplit{133083032074604559604340287117839089926377586695643969929960}\\
18, 62 & \seqsplit{4914973169426146305547121237115076142934306590762664470109465}\\
19, 61 & \seqsplit{144790133496718992729183963898742898381958489364050509826157850}\\
20, 60 & \seqsplit{3444406396738331940042203205261748448618077608990920460297152927}\\
21, 59 & \seqsplit{66879429632522158555831061193306393666091121818246595019973494540}\\
22, 58 & \seqsplit{1069879245233979553273989981121291380614464860405669910379850081260}\\
23, 57 & \seqsplit{14216931291161213471799362567564672157951957568025594179045791891560}\\
24, 56 & \seqsplit{158068390384168461407275178099958985150848615404595734114695885650960}\\
25, 55 & \seqsplit{1479869131854343720221412587281665870701498682700938794266267447170086}\\
26, 54 & \seqsplit{11732508966513353501037129922747065430335778896134680507066470093403965}\\
27, 53 & \seqsplit{79161884227224246508679052646403404345905368341688413538814635813598960}\\
28, 52 & \seqsplit{456581043511566566513756838750649299641150922784219665003799193406125100}\\
29, 51 & \seqsplit{2259917455015013347676040286687575994178632441330552600666858713130898160}\\
30, 50 & \seqsplit{9632450231336794527947449205326436851306868754765035458406149555236620420}\\
31, 49 & \seqsplit{35462233590117554296833075539562847955623515746945956607220659491897784880}\\
32, 48 & \seqsplit{113065755847885865404818291759672029697970400130807039565248284647947993030}\\
33, 47 & \seqsplit{312918069826522600854890638200206101721857265579850116975761027358610292700}\\
34, 46 & \seqsplit{753222026961378730490689274083776842139584094531888612894098822340268660325}\\
35, 45 & \seqsplit{1579556357009619711400936241296625724112698332853761904741610358853299545374}\\
36, 44 & \seqsplit{2889815889806322306439673258686027236373634776675525004624936294928758513525}\\
37, 43 & \seqsplit{4617572925631922908016027106398595888158183324096994801079581389319797290920}\\
38, 42 & \seqsplit{6449688740325452249773160861843840833404355457938341602310966703875348540240}\\
39, 41 & \seqsplit{7879686549589815648064051919266271930294598739569578564451735742954096193580}\\
40 & \seqsplit{8423286746547510694781471140447820128457705048915915535966580408775575080416}\\
\end{longtable}
\endgroup

\clearpage
\begingroup
\small
\setlength{\tabcolsep}{5pt}
\renewcommand{\arraystretch}{1.05}
\begin{longtable}{@{}>{\raggedleft\arraybackslash}p{0.075\textwidth}>{\ttfamily\fontsize{7.6}{10}\selectfont\raggedright\arraybackslash}p{\dimexpr0.925\textwidth-2\tabcolsep\relax}@{}}
\caption{Ehrhart series of $\mathcal W_{12}$: $\Ehr_{\mathcal W_{12}}(z)=g_{12}(z)/(1-z)^{110}$ with $\deg_z g_{12}(z)=99$.}\label{tab:wc12}\\
\toprule $j$ & \multicolumn{1}{l}{$b_{12,j}$}\\\midrule\endfirsthead
\multicolumn{2}{l}{\small Table \thetable\ continued ($n=12$).}\\
\toprule $j$ & \multicolumn{1}{l}{$b_{12,j}$}\\\midrule\endhead
\midrule\multicolumn{2}{r}{\small Continued on the next page.}\\\endfoot
\bottomrule\endlastfoot
0, 99 & \seqsplit{1}\\
1, 98 & \seqsplit{176214731}\\
2, 97 & \seqsplit{8473000154156285}\\
3, 96 & \seqsplit{30860312127211613449975}\\
4, 95 & \seqsplit{15785066810925440592158445015}\\
5, 94 & \seqsplit{1798916967828211452711661968310965}\\
6, 93 & \seqsplit{62744393132928999194280374661715931355}\\
7, 92 & \seqsplit{838123495894683741639035400727858854022473}\\
8, 91 & \seqsplit{5049269971218029592439995097223784763034360010}\\
9, 90 & \seqsplit{15510210231568851186483393251966217064742414104870}\\
10, 89 & \seqsplit{26693829368344514257522283383640088576094469930897366}\\
11, 88 & \seqsplit{27714066094550643160281468665049313679733999440443331246}\\
12, 87 & \seqsplit{18411402046626999621414405617335458299484925725072866613011}\\
13, 86 & \seqsplit{8209319014002253387784114737398847481236242693229378326118817}\\
14, 85 & \seqsplit{2554963651413043175322723627801719767547462190540159558089284775}\\
15, 84 & \seqsplit{573407800001862476417032639754666211032903406355330242380511022069}\\
16, 83 & \seqsplit{95372130711616597570235197661879350917928148455252676839898989814856}\\
17, 82 & \seqsplit{12031758520903761074401829050659749387286782425309052839970426917990596}\\
18, 81 & \seqsplit{1174353601787210245218868416385235914493139560440124405935370305529049424}\\
19, 80 & \seqsplit{90209720996799744561548436076757834981855934080426177172680567738879677964}\\
20, 79 & \seqsplit{5535210958596593437503665131096514613175938260081710205424740760691041998033}\\
21, 78 & \seqsplit{274831804727761113673078771458621890438269366276936549402909331398082473956971}\\
22, 77 & \seqsplit{11168435807511614890021825631162870460048912666538188215694628551744164385311261}\\
23, 76 & \seqsplit{375207201606298170521984273166343775781388563679801220694021440568008988594958871}\\
24, 75 & \seqsplit{10514050269768012074760522867654161953771708683835135726330677038747441645736498962}\\
25, 74 & \seqsplit{247702926006901951667257265523311311936950962128919336722648772899665467039310584030}\\
26, 73 & \seqsplit{4941167200790384169556309972131730218719521359341933503026079530846720591212264774186}\\
27, 72 & \seqsplit{83988738922827591899814515949841674190024172956991033553397415849335579388212697581826}\\
28, 71 & \seqsplit{1223437732957417158139368092451740586758092374711547274912626185667211947640612893316813}\\
29, 70 & \seqsplit{15351188854231886446352323521687617564262888233569773107143173891244374926724569463202943}\\
30, 69 & \seqsplit{166691537968319832207528349609456969861492113148459391659980833051955045845571013274806073}\\
31, 68 & \seqsplit{1572948491995967709644641163657083842849455733582834613726215238329558989925054962488708187}\\
32, 67 & \seqsplit{12947598150228888107375802098655017806837903085872673235422156144584689176842410436679037605}\\
33, 66 & \seqsplit{93288009611604185366652218389524822809062216124775217612890569677980245441345120906610463547}\\
34, 65 & \seqsplit{590162024100272358174916861449363673014045269500012563539315231285671100595041573554646708309}\\
35, 64 & \seqsplit{3287336039164925219570532745382660166461267049671245793610709338584828197672860746802306014351}\\
36, 63 & \seqsplit{16163885408291763532270430998374091663740533327878983488724186972462521359982325382634053579622}\\
37, 62 & \seqsplit{70318585131974503183900592911190435861432529331275136274858641416070583813205698831967229129610}\\
38, 61 & \seqsplit{271213859775401520836885909753242859616230482059987042683034391848508801789498833182382342887318}\\
39, 60 & \seqsplit{929124222297904024001272883349970937397945207602424049045659025641889722102831458169068905093466}\\
40, 59 & \seqsplit{2831867387916240132542449613034628941222639940574741181612679459479057549620046824848526826747200}\\
41, 58 & \seqsplit{7690349875395516080770866925756937817294256399945087098699812962689274533211206779138492146156400}\\
42, 57 & \seqsplit{18631790718536674374541884777221847299973672405629065258854592691243624282878708887306781829663684}\\
43, 56 & \seqsplit{40316990079673636590639956692448225179909794046535234279965560156505634326981189926189507520025516}\\
44, 55 & \seqsplit{77995251257185062228095897091514806509404671321240523424885498343568750340349808434513395086910966}\\
45, 54 & \seqsplit{135005375935547922186540950817761119014405924790055081175818691841538563447737787627872571101131882}\\
46, 53 & \seqsplit{209233017530279457387089455235931853727992984952295787875813264544251871517766438238505203690186694}\\
47, 52 & \seqsplit{290494629986038446778555859350445333879380307846853607929689490982859227571746175364205167797385146}\\
48, 51 & \seqsplit{361449638327453741869304128776491899974012165062697712897457443767246385632363329980351638717487760}\\
49, 50 & \seqsplit{403156845449265351112640821819662132768277398026233006513350971419536282277133486428458996489364484}\\
\end{longtable}
\endgroup

\clearpage
\begingroup
\small
\setlength{\tabcolsep}{5pt}
\renewcommand{\arraystretch}{1.05}
\begin{longtable}{@{}>{\raggedleft\arraybackslash}p{0.075\textwidth}>{\ttfamily\fontsize{7.6}{10}\selectfont\raggedright\arraybackslash}p{\dimexpr0.925\textwidth-2\tabcolsep\relax}@{}}
\caption{Ehrhart series of $\mathcal W_{13}$: $\Ehr_{\mathcal W_{13}}(z)=g_{13}(z)/(1-z)^{132}$ with $\deg_z g_{13}(z)=120$.}\label{tab:wc13}\\
\toprule $j$ & \multicolumn{1}{l}{$b_{13,j}$}\\\midrule\endfirsthead
\multicolumn{2}{l}{\small Table \thetable\ continued ($n=13$).}\\
\toprule $j$ & \multicolumn{1}{l}{$b_{13,j}$}\\\midrule\endhead
\midrule\multicolumn{2}{r}{\small Continued on the next page.}\\\endfoot
\bottomrule\endlastfoot
0, 120 & \seqsplit{1}\\
1, 119 & \seqsplit{2290792800}\\
2, 118 & \seqsplit{1373423575465766022}\\
3, 117 & \seqsplit{46761069504313551048368212}\\
4, 116 & \seqsplit{177717627840988922046848177626365}\\
5, 115 & \seqsplit{125286711088199103660209300217347485476}\\
6, 114 & \seqsplit{23280987154160190740620492721206182368143744}\\
7, 113 & \seqsplit{1463789501322354004938422118453828896985271898640}\\
8, 112 & \seqsplit{37411986865556785033643799520027894364414847641040804}\\
9, 111 & \seqsplit{446501837878810385355626113754265425463255742801522352996}\\
10, 110 & \seqsplit{2770325294293337557624899853479478415773504804837913865521190}\\
11, 109 & \seqsplit{9725916072190564938573073384549352945019435316283020408062081024}\\
12, 108 & \seqsplit{20680326730383196019598281441428171713662621903029703435298648155789}\\
13, 107 & \seqsplit{28148790051545642188913069794634183634817525408736359750887244984527528}\\
14, 106 & \seqsplit{25672378594662851391615486286833019734752828982614007086623780131619967766}\\
15, 105 & \seqsplit{16297058437945132533399231735119876447304750969093179984207923461783120605876}\\
16, 104 & \seqsplit{7435536554835146193922204924051351294369430369102141450660510538870803273023644}\\
17, 103 & \seqsplit{2505575173578136106032109791776067645637433401145460526227464482357063652788026632}\\
18, 102 & \seqsplit{638311006999916770915638692870489574715763907918175636983880187076416432112254382842}\\
19, 101 & \seqsplit{125440606214092146564164576622159480181329520307659682475899299068296092946929383415764}\\
20, 100 & \seqsplit{19352217649879031878414126129286418667028871323956979721095267314702516170812454092119035}\\
21, 99 & \seqsplit{2379917687744536026365588476482237706177554502073547145424486262039836654656539011568673388}\\
22, 98 & \seqsplit{236473489947966848445527123330908351826234724794486472998107782427474661472874939226850831672}\\
23, 97 & \seqsplit{19211682152577116099797873355201179986987653738699761383398989712523436364717035686780819448104}\\
24, 96 & \seqsplit{1289753190728493238432311626421387813624411506040803853756079183872731526431891443661573676215994}\\
25, 95 & \seqsplit{72228258735976660315274924006744967846191773500874911012130253431694696643563199571025455836304608}\\
26, 94 & \seqsplit{3402837296002875595212531762876074639267031073246758062434970433610493314936786646711646794547990438}\\
27, 93 & \seqsplit{135897858463239413874439990402054151306687029398548395433166133997414925505855612060292576893713349072}\\
28, 92 & \seqsplit{4632360079503061875143423946571694000396221994008476360217180967901080880573492360975359560548942599991}\\
29, 91 & \seqsplit{135615069360677931357819771901661995844364375389725718745255634636013229980462008468560308804159334265380}\\
30, 90 & \seqsplit{3429080517631362192665613504347144718963250242041029275526527239432683943964835257693297364447402301412046}\\
31, 89 & \seqsplit{75272922782807567390200369315365232234745268823739994641517255324129362751096303246087956628963250259354916}\\
32, 88 & \seqsplit{1441192924440153393162354140930144621993604605580493120028551554294626273506154990952744744704043779654310681}\\
33, 87 & \seqsplit{24170365665504633539934951394621958449799101416747025272514221700587111570502108202134275470228501171795836288}\\
34, 86 & \seqsplit{356467838750637016676344291482271300595711172793082176496347084207411044727704199935301009766096599051791366222}\\
35, 85 & \seqsplit{4639703740228613763967870492558317229993897403485426711823800140807473188142553105409518463698792187371424095540}\\
36, 84 & \seqsplit{53471422887894823188235861004081558088623603906687575320821344259066200321212918184946328309083706323566633207325}\\
37, 83 & \seqsplit{547302763504697432050287439829710731681073624494214014682978070566604456105031085458933401457058348125575271955956}\\
38, 82 & \seqsplit{4989018993542030646941547610998482470222100324136690114604909275061712361812760979974368208481325886830991281580396}\\
39, 81 & \seqsplit{40606448989849317678260588571289977210257198721921089444552047183033539185696012796304559385955857904466204884997000}\\
40, 80 & \seqsplit{295794075320001477930967864593880999716361845638071719842964644012747950355195316345155626415457208620080909236491297}\\
41, 79 & \seqsplit{1932588614227247204315616755191398531635666805568833730772314037509964853538658817386658542901481208640059612453691024}\\
42, 78 & \seqsplit{11347772041562575644595010805222380742282035969265884088842037955572883028199439818325443771290139615410784902933083246}\\
43, 77 & \seqsplit{59992807509758138817572961043154509332006366491762274128454996250829103024995238258603747397354750205853990359187237780}\\
44, 76 & \seqsplit{286046493584950904041143916790196838453766042943977007595490509639831476679363541858320715646215490425434090412552854821}\\
45, 75 & \seqsplit{1231952736114744107722463324464023927030229550181521426105488167136534006927968069071579657164955155616769562444657473860}\\
46, 74 & \seqsplit{4799390341747372585557332194777233909051054560645798015620170116675337425752402599427503186370958289437250431871128360418}\\
47, 73 & \seqsplit{16934574368113041662874957351845045396817789432540694261233184524901309446864708069249873200884771295486474555811677528376}\\
48, 72 & \seqsplit{54183898141350243727692433347317279333564419360356878269674762026884178793390675188550545115435120270942620748560077998335}\\
49, 71 & \seqsplit{157375890237043600906257237251854783475013121911022137795148267593119928010597451351133708484679695351952251899466392689280}\\
50, 70 & \seqsplit{415333976284640754644776504162748478551906161170182316006051089783479831830019742570674260700372701505167021272294495114454}\\
51, 69 & \seqsplit{996837836358083770794492235459758330887101515096439701283713856317763008977341435850199699264483451126351661071129225327668}\\
52, 68 & \seqsplit{2177483504607708748061212877692476807311155760748474203619916685434649059477025440484269133105666925728924120906531141067135}\\
53, 67 & \seqsplit{4331962548810434399067863582587484748414210339518022408022926467231449278652072977883358552232391638424449598889048812368684}\\
54, 66 & \seqsplit{7853661366609763573741490057077253501544844495661861036970799918291994479518975349385129266989522520832776057124155806543520}\\
55, 65 & \seqsplit{12981914835379315469414318365962886686980919301726483702128515157604926299506507533912317181147318161433927686821212882781656}\\
56, 64 & \seqsplit{19573575327755675433223805198179974820104431803163650344835141027126302555423861586972083343052793048991329752243961684026904}\\
57, 63 & \seqsplit{26928813381674572358643647222127850596692242353275172068473254342883632636624751880234594414040385866933238354961624260883884}\\
58, 62 & \seqsplit{33813999369449415974517758695457317508448153602499747129122885163512767178160192542354150699668837427493983051382844721206488}\\
59, 61 & \seqsplit{38760651131856467205922012034755782138839199623356219569544412600127960062900606547367887509690039870991960933720248927837676}\\
60 & \seqsplit{40564899179363664031790050613664782979175646028575412682778483218013219527917160559583817135836769093646097869540741392227994}\\
\end{longtable}
\endgroup

\end{document}